\documentclass[10pt]{amsart}
\usepackage{amsfonts}
\usepackage{amsmath}
\usepackage{fancyhdr,amssymb}
\usepackage{amsthm}
\usepackage{arydshln}
\usepackage{ color, ulem}
\newcommand{\bigzero}{\mbox{\normalfont\Large 0}}
\newcommand{\rvline}{\hspace*{-\arraycolsep}\vline\hspace*{-\arraycolsep}}
\input xy
\xyoption {all}

\DeclareMathOperator{\arcsinh}{arcsinh}

\DeclareMathOperator{\Hom}{Hom}
\DeclareMathOperator{\cO}{\mathcal{O}}
\DeclareMathOperator{\quot}{\mathsf{Quot}}
\DeclareMathOperator{\fix}{\mathrm{F}}
\newcommand{\ch}{\textnormal{ch}}
\newcommand{\td}{\textnormal{Td}}
\newcommand{\comment}[1]{}
\newtheorem{theorem}{Theorem}
\newtheorem {lemma}[theorem]{Lemma}

\newtheorem{question}[theorem]{Question}
\newtheorem {corollary}[theorem]{Corollary}

\theoremstyle{definition}
\newtheorem {example}[theorem]{Example} 
\theoremstyle {definition}

\begin{document}

\baselineskip=16.75pt
\title[Euler characteristics over Quot schemes]{Euler characteristics of tautological bundles over Quot schemes of curves}
\author{Dragos Oprea}
\address{Department of Mathematics, University of California, San Diego}
\email {doprea@math.ucsd.edu}
\author{Shubham Sinha}
\address{Department of Mathematics, University of California, San Diego}
\email {shs074@ucsd.edu}

\begin{abstract} We compute the Euler characteristics of tautological vector bundles and their exterior powers over the Quot schemes of curves. We give closed-form expressions over punctual Quot schemes in all genera. For higher rank quotients of a trivial vector bundle, we obtain answers in genus $0$. We also study the Euler characteristics of the symmetric powers of the tautological bundles, for rank $0$ quotients. 
\end{abstract}
\maketitle
\section{Introduction}
In this paper, we prove several closed-form expressions for the holomorphic Euler characteristics of tautological vector bundles and their exterior and symmetric powers over the Quot schemes of curves. 

\subsection{Punctual Quot schemes} To set the stage, let $E\to C$ be a locally free sheaf over a smooth projective curve $C$. Let $\quot_d(E)$ denote the Quot scheme parameterizing rank $0$ degree $d$ quotients of $E$: $$0\to S\to E\to Q\to 0,\quad \text{rank }Q=0, \quad \text{deg }Q=d.$$ It is easy to see that $\quot_d(E)$ is smooth of dimension $Nd$ where $N=\text{rank }E.$ 

We write $$0\to \mathcal {S}  \to p^*E \to \mathcal Q\to 0$$ for the universal exact sequence over $C\times \quot_d(E),$ and we let $p$ and $\pi$ denote the two projections over the factors of $C\times \quot_d(E)$. For any line bundle $L\to C$, there is an induced tautological rank $d$ vector bundle over $\quot_d(E)$ given by \[L^{[d]}=\pi_*(p^* L\otimes \mathcal Q). \] 
	
	We first study the holomorphic Euler characteristics of all exterior powers $\wedge^k L^{[d]}$. To this end, for any vector bundle $V$ over a scheme $Y$, we set $$\wedge_y V:=\sum_k y^k \wedge^kV.$$ We show 
	\begin{theorem}\label{t1}
		Let $E\to C$ be a vector bundle over a smooth projective curve,  and let $L\to C$ be a line bundle. Then 
		$$\sum_{d=0}^{\infty} q^d\chi(\quot_d(E), \wedge_yL^{[d]})={(1-q)^{-\chi(\cO_C)}}{(1+qy)^{\chi(E\otimes L)}}.$$	\end{theorem}	
		
\noindent The same methods will establish a slightly stronger result:
	\begin{theorem}\label{t2}
		For any line bundles $M_1,M_2,\dots, M_r$ and $L$ over $C$, where $0\le r\le \text{rk } E-1$, we have 
		\begin{equation*}
			\sum_{d=0}^{\infty} q^d\chi\left(\quot_d(E), \wedge_yL^{[d]}\otimes_{i=1}^{r}\big(\wedge_{x_i} M_i^{[d]} \big)^{\vee}\right)=(1-q)^{-\chi(\cO_C)} (1+qy)^{\chi(E\otimes L)} \prod_{i=1}^{r}(1-qx_iy)^{-\chi(M_i^{\vee}\otimes L)}.
		\end{equation*}
	\end{theorem}
	
Our proofs rely on universality arguments in the spirit of \cite {EGL} to reduce to the case of genus $0$, and equivariant localization in genus $0$. These are well-established techniques. In general however, the ensuing localization sums are not immediately expressed in closed form. The novelty here is that we show how to overcome the combinatorial difficulties by using Lagrange-B\"urmann inversion and considerations involving Schur polynomials and Jacobi-Trudi determinants. These lead to drastic simplifications of the answers. 

In a similar context, the prior work \cite {MO} also makes use of localization techniques to calculate a large part of the intersection theory of Quot schemes of curves (for quotients of arbitrary ranks of a trivial vector bundle), in particular to derive and extend the Vafa-Intriligator formula. The current $K$-theoretic setup requires that we handle the localization sums differently. We believe that the simplicity of the final formulas makes the calculations worthwhile to be recorded. 

\subsection{Analogies with surfaces}	 \label{aws}Theorems \ref{t1} and \ref{t2} suggest an unexpected analogy between $\mathsf{Quot}_d(E)$ and the Hilbert scheme of points $X^{[d]}$ over a smooth projective surface $X$. Specifically, Theorem 1 can be compared with the calculations of \cite{A, Da, Sc}: $$\sum_{d=0}^{\infty} q^d \chi(X^{[d]}, \wedge_y L^{[d]}) =(1-q)^{-\chi(\mathcal O_X)} (1+qy)^{\chi(L)}.$$ This is proved in \cite {Sc} by passing to the derived category and computing the image of $\wedge^\bullet L^{[d]}$ under the Bridgeland-King-Reid equivalence $${\mathbf D}^{\textrm{b}}(X^{[d]})\simeq \mathbf{D}^{\textrm{b}}_{S_d}(X^d)$$ established in \cite {BKR, H}. The same result is obtained by studying different equivariant limits in the Donaldson-Thomas theory of toric Calabi-Yau $3$-folds in \cite{A}.

In the same vein, Theorem 2 when $r=1$ mirrors the following result of \cite{WZ} and its strengthening by \cite {K}: \begin{align*}\sum_{d=0}^{\infty} q^d \chi\left(X^{[d]}, \wedge_{y} L^{[d]}\otimes (\wedge_{x} M^{[d]})^{\vee}\right)&=(1-q)^{-\chi(\mathcal O_X)} (1+qy)^{\chi(L)}(1+qx)^{\chi(M^{\vee})}(1-qxy)^{-\chi(L\otimes M^{\vee})}.%\\&=\exp \left(\sum_{d=1}^{\infty} \frac{q^d}{d} \chi\left(\wedge_{-x^d} M, \wedge_{-y^d} L\right) \right).
\end{align*} 
While the latter equality is conjectured to hold in all dimensions $\geq 2$, it was noted in \cite {K} that the naive analogue for symmetric products of curves fails. Theorem \ref{t2} can be viewed as a remedy: a similar but {\it different} formula holds for curves, and only when $N\geq 2$. Theorem \ref{t2} furthermore allows for multiple dualized factors. 

\subsection{Higher rank} Turning to higher rank quotients, while we do not obtain a closed-form expression for the corresponding generating series, our techniques yield the result below. 
We restrict to the case $C=\mathbb P^1$, and $E$ a trivial rank $N$ vector bundle. The Quot scheme $\quot_d(E, r)$ parametrizes short exact sequences $$0\to S\to E\to Q\to 0, \quad \text{rank }Q=r, \quad \deg Q=d.$$ 

	\begin{theorem}\label{t3} Let $\deg L=\ell$ and $0<r<N$. We have 	
	\begin{align*}
			\chi(\quot_d(E, r),\wedge_y L^{[d]}) =(-1)^{(N-r-1)d}\left[q^d\right]\frac{\det(f_i(z_j))}{\det (z_i^{N-j})}.
	\end{align*} 
In the numerator $(f_i(z_j))$ is the $N\times N$ matrix with \begin{align*}
	f_i(z)=\begin{cases}
		z^{\ell+d+N-i+1}& \text{ if }1\le i\le N-r\\
		z^{N-i}(z+y)^{\ell+1}& \text{ if } N-r+1\le i\le N
	\end{cases}
\end{align*}
and $z_1,\dots,z_N$ are the distinct roots of the equation $(z-1)^N-q(z+y)z^{r-1}=0$. The denominator is the Vandermonde determinant. 
\end{theorem}
\noindent In the statement above, the brackets indicate taking the coefficient of the corresponding power of $q$. 

In Corollary \ref{thm:high_rank_det} of Section \ref{s4}, we note a connection between $\chi\left(\quot_d(E, r) ,\det L^{[d]}\right)$ and Schur polynomials. We also work out several specializations of the corresponding formula. 

\subsection{Symmetric products} For rank zero quotients, we similarly study the series of symmetric powers of the tautological bundles $$\sum_{d=0}^{\infty} q^d \chi\left(\quot_d(E), \mathsf{Sym}_y L^{[d]}\right).$$ Here, for any vector bundle $V$, we write$$\quad \mathsf{Sym}_yV=\sum_{k=0}^{\infty} y^k\, \mathsf{Sym}^kV.$$ 

\begin{theorem}\label{t4} For $C=\mathbb P^1$ and $d\geq k$,  we have $$\chi\left(\quot_d(E), \mathsf {Sym}^k L^{[d]}\right)=\binom{\chi(E\otimes L)+k-1}{k}.$$ \end{theorem} 
\noindent The answer is independent of $d\geq k$. The question was also studied over Hilbert schemes of points on surfaces \cite {A, Sc2}. Just as for exterior powers, the analogy with the Hilbert scheme of surfaces persists here as well. Indeed, the following result was established in \cite {A} via the Donaldson-Thomas theory of toric Calabi-Yau $3$-folds $$\chi\left(X^{[d]}, \mathsf {Sym}^k L^{[d]}\right)=\binom{\chi(L)+k-1}{k},$$ whenever $d\geq k$ and $\chi(\mathcal O_X)=1$. 

Theorem \ref{t10} in Section \ref{s3} gives a more general result than Theorem \ref{t4} above, covering the case $d<k$ as well. 

In arbitrary genus, universality arguments as in \cite {EGL} show that the factorization $$\mathsf{W}=\sum_{d=0}^{\infty} q^d \chi\left(\quot_d(E), \mathsf{Sym}_y L^{[d]}\right)=\mathsf A^{\chi(\mathcal O_C)} \cdot \mathsf B^{\chi(E\otimes L)}$$ holds true, for two universal power series $\mathsf A, \mathsf B\in \mathbb Q(y)[[q]]$ that depend on $N$, but not on the triple $(C, E, L)$. Our results give precise information about the series $\mathsf B$. While we can determine $\mathsf A$ in principle, we do not have a closed-form expression.  

\begin{theorem} \label{t5}We have $$\mathsf B=f\left(\frac{qy}{(1-y)^{N+1}}\right)$$ where $f(z)$ is the solution to the equation $$f(z)^N-f(z)^{N+1}+z=0,\,\quad f(0)=1.$$
\end{theorem}

\noindent For instance, in the special case $N=2$, we obtain $$f(z)=1+\frac{4}{3} \sinh^2\left(\frac{1}{3}\arcsinh\left(\frac{3\sqrt{3z}}{2}\right)\right).$$

In the last section, we raise a few questions related to the results in this paper. 
\subsection {Previous work.} 

Several aspects of the geometry of the punctual Quot schemes $\mathsf{Quot}_d(E)$ are already understood. Since it is difficult to be exhaustive here, we mention only a few selected results. The Poincar\'e polynomials and the motives of $\mathsf{Quot}_d(E)$ are computed in \cite{B, BFP, C, R, St}, the cohomology and stabilization phenomena are considered in \cite{M}, automorphisms and Torelli type results are studied in \cite{BDH, Ga}, various cones of divisors are investigated in \cite{GS, St}, positivity results for the tautological vector bundles are obtained in \cite{O}, the cohomology of the tangent bundle is studied in \cite{BGS}. The punctual Quot schemes bear close connections to the moduli space of bundles over curves, and in fact, the study of the Poincar\'e polynomials and motives of the latter can be undertaken in this context \cite{BD, BGL, HPL}. As alluded to above, the (virtual) intersection theory of Quot schemes (of quotients of arbitrary rank) over curves was studied in a wider context, see \cite{Ber1, Ber2, MO}. Furthermore, in this wider context, connections with the intersection theory of the moduli of bundles were proven in \cite{BDW, Ma, MarOp}. 

Regarding the $K$-theory of Quot schemes, fewer direct calculations were available. \footnote{Over surfaces, the virtual $K$-theory of the punctual Quot schemes $\mathsf{Quot}_d(E)$ was thoroughly studied in \cite{AJLOP, Bo}. Via cosection localization to a canonical curve, these results can be used to compute certain {\it twisted} holomorphic Euler characteristics of tautological bundles over the punctual Quot schemes. These calculations are indirect, and more importantly, due to the aforementioned twists, the results of our paper are not accessible via these methods. } We note however the work of \cite{RZ} which connects certain holomorphic Euler characteristics of {\it line} bundles over certain Quot schemes (parametrizing rank 2 subbundles of a trivial high rank bundle) with the Verlinde numbers over the moduli of rank $2$ bundles, via a wallcrossing analysis as in \cite{BDW, T}. \footnote{It is natural to ask whether the Euler characteristic $\chi(\wedge^d L^{[d]})$ given by Theorem \ref{t1} can be computed via wall-crossing, say in rank $2$ for simplicity. In principle, we are led to finding the Euler characteristics of certain higher rank twists of the theta line bundles over the moduli space of semistable vector bundles, whose intersection theory is known. However, obtaining exact formulas in this fashion seems combinatorially difficult.} 

\subsection{Acknowledgements} We thank Noah Arbesfeld, Patrick Girardet, Drew Johnson, Woonam Lim, Alina Marian, Rahul Pandharipande and Ming Zhang for discussions related to $K$-theoretic invariants of Quot and Hilbert schemes. The authors are supported by the NSF through grant DMS 1802228. 	 
\section {Rank zero quotients and exterior powers}
In this section we prove Theorems \ref{t1} and \ref{t2}. The proofs of the two theorems are similar. The calculations for Theorem \ref{t1} are however simpler and already illustrate the main points. 
	\label{section2}
	\subsection {Torus action} \label{torusact}
		We first establish Theorem \ref{t1} when $C=\mathbb P^1$ and $E=\cO(a_1)\oplus\cdots \oplus \cO(a_N)$, and for $L$ such that $$\deg L+a_i+1\geq 0$$ for all $i$. The arbitrary genus case follows from here by universality arguments, see Section \ref{arb} below. 
	
	Under the above assumptions, we seek to show that \begin{equation}\label{answ}
		\sum_{d=0}^{\infty} q^d\chi(\quot_d, \wedge_yL^{[d]})=(1-q)^{-1}(1+qy)^{\chi(E\otimes L)}.
	\end{equation} 
Here, for simplicity, we wrote $\quot_d$ instead of $\quot_d(E)$. 

	We evaluate expression \eqref{answ} via Hirzebruch-Riemann-Roch 
	$$\chi\left(\quot_d, \wedge_y L^{[d]}\right)=\int_{\quot_d} \ch (\wedge_y L^{[d]})\, \td \left(\quot_d\right),$$ and we use $\mathbb{C}^*$-equivariant localization to compute the integral. \footnote{It is natural to attempt localization directly in $K$-theory, but we were unable to establish the result in this fashion.}

To this end, we let $\mathbb C^*$ act on $E$ with weight $-w_i$ on the summand $\cO(a_i)$. This induces a $\mathbb C^{\star}$-action on $\quot_d$. The fixed subbundles correspond to split inclusions $$S=\bigoplus_{i=1}^{N} K_i(a_i)\hookrightarrow E=\bigoplus_{i=1}^{N} \mathcal O(a_i).$$ Thus,  
the fixed loci are products of projective spaces $$\fix_{\vec{d}}=\mathbb{P}^{d_1}\times\cdots\times \mathbb{P}^{d_N}$$ for vectors $\vec d=(d_1, \ldots, d_N)$ such that $d_1+\cdots +d_N=d$. The factor $\mathbb{P}^{d_i}$ corresponds to the Hilbert scheme of $d_i$ points of $\mathbb{P}^1$ parameterizing short exact sequences $$0\to K_i\to \cO\to T_i\to 0$$ such that $T_i$ is a torsion sheaf of length $d_i$. 

There is a universal exact sequence
	\[0\to \mathcal{K}_i\to \cO\to \mathcal{T}_i\to 0  \] over the product $\mathbb{P}^1\times \mathbb{P}^{d_i}$, 
	 with the universal kernel given by \begin{equation}\label{ki}\mathcal{K}_i = \cO_{\mathbb{P}^1}(-d_i)\boxtimes \cO_{\mathbb{P}^{d_i}}(-1).\end{equation}
For future reference, we note that the universal exact sequence $0\to \mathcal {S}\to p^*E\to \mathcal Q\to 0$ over $\mathbb P^1\times \quot_d$ restricts to $\mathbb P^1\times \fix_{\vec{d}}$ as 
	\begin{equation}\label{e5}
		0\to \bigoplus_{i\in [N]} \mathcal{K}_i(a_{i})\to \cO(a_{1})\oplus\cdots \oplus \cO(a_{N})\to  \bigoplus_{i\in [N]}\mathcal{T}_i(a_{i})\to 0,
	\end{equation}
where pullbacks from the factors are understood above. We also set $[N]=\{1, 2, \ldots, N\}.$

	By Atiyah-Bott localization, we have \begin{equation}\label{ab}
		\chi(\quot_d,\wedge_y L^{[d]}) = \sum_{|\vec{d}|=d}\int_{\fix_{\vec{d}}}
		\ch(\wedge_y L^{[d]})\frac{\td(\quot_d)}{e_{\mathbb{C}^*}(\textrm {N}_{\vec{d}})}\bigg|_{\fix_{\vec{d}}}. 
	\end{equation}
Here $\textrm{N}_{\vec d}$ denotes the normal bundle of the fixed locus $\textrm{F}_{\vec d}$.

\subsection{Explicit calculations} We proceed to calculate the expressions appearing in the localization sum \eqref{ab}. In the next subsections, we record the Todd genera, the normal bundle contributions and the Chern characters of the tautological bundles. 
	\subsubsection{Todd Classes}  By \eqref{e5}, the tangent bundle $T{\quot_d}=\Hom_\pi (\mathcal {S},\mathcal Q)$ restricts to \[\bigoplus_{i, j\in [N]}\pi_*\left(\mathcal{K}^{\vee}_i(-a_{i}) \otimes\mathcal{T}_j(a_{j})\right) \] over the fixed locus $\fix_{\vec{d}}= \mathbb{P}^{d_1}\times\cdots\times \mathbb{P}^{d_N}$. Here $\pi:\quot_d\times \mathbb P^1\to \quot_d$ denotes the projection. In $K$-theory, the above expression equals \[\bigoplus_{i,j\in [N]}\pi_*\left(\mathcal{K}_i^{\vee}(a_j-a_{i})\right)-\bigoplus_{i,j\in [N]}\pi_*\left(\mathcal{K}^{\vee}_i \otimes\mathcal{K}_j(a_{j}-a_{i})\right) .\] In this discussion $\pi_{\star}=R^0\pi_{\star}-R^1\pi_{\star}.$ Therefore the Todd class of ${\quot_d}$ restricted to each fixed locus is \[\prod_{i,j\in [N]}\td\left(\pi_*\left(\mathcal{K}^{\vee}_i(a_j-a_{i})\right)\right)\left(\prod_{i,j\in [N],\,\,i\neq j}\td\left(\pi_*\left(\mathcal{K}^{\vee}_i \otimes\mathcal{K}_j(a_{j}-a_{i})\right)\right)\right)^{-1}. \] The above $(i, j)$-terms carry the weight $w_{i}-w_j$. The assumption $i\neq j$ in the second product can be made since the term $i=j$ is trivial in genus $0$.

	\subsubsection{Equivariant normal bundles} Over each fixed locus, the normal bundle is given by the moving part of the tangent bundle:
	\begin{align}\label{movingpart}
		\textrm{N}_{\vec {d}}=T^{\text{mov}}\bigg|_{\fix_{\vec{d}}}&= \bigoplus_{ i\ne j}\pi_*\left(\mathcal{K}^{\vee}_i(-a_{i}) \otimes\mathcal{T}_j(a_{j})\right)\\
		&= \bigoplus_{i\ne j}\pi_*\left(\mathcal{K}_i^{\vee}(a_j-a_{i})\right)-\bigoplus_{i\ne j}\pi_*\left(\mathcal{K}^{\vee}_i \otimes\mathcal{K}_j(a_{j}-a_{i})\right),\nonumber
	\end{align}
	where we continue to keep track of the weights $w_{i}-w_j$. Therefore, we find the Euler classes  
	\begin{align*}
		\frac{1}{e_{\mathbb{C}^*}(\textrm{N}_{\vec d})}=\prod_{i, j\in [N], \,\,i\ne j } \bigg(e_{\mathbb{C}^*}\left(\pi_{\star}\left(\mathcal{K}_i^\vee(a_j-a_{i})\right)\right) \bigg)^{-1}\prod_{i,j\in [N], \,i\ne j}e_{\mathbb{C}^*}\left(\pi_{\star}\left(\mathcal{K}^{\vee}_i \otimes\mathcal{K}_j(a_{j}-a_{i})\right)\right).
	\end{align*}
	
	Collecting all expressions above, we obtain that over the fixed locus $\fix_{\vec{d}}$, the factor $\frac{\td (\quot_d)}{e_{\mathbb{C}^*}(\textrm{N}_{\vec d})}$ in the localization expression \eqref{ab} restricts to
	\begin{equation}\label{toddeu}
		\prod_{i\in[N]}\td\left(\pi_* \left(\mathcal{K}_i^\vee\right)\right) \prod_{i,j\in [N],\,\, i\neq j}\frac{\td}{e_{\mathbb{C}^*}}\bigg(\pi_*\left(\mathcal{K}^{\vee}_i(a_j-a_{i})\right)\bigg)\prod_{i, j\in [N],\, i\ne j}\frac{e_{\mathbb{C}^*}}{\td}\bigg(\pi_{\star}\left(\mathcal{K}^{\vee}_i \otimes\mathcal{K}_j(a_{j}-a_{i})\right) \bigg).
	\end{equation}
	\vskip.1in
	\subsubsection{Explicit contributions} The terms in \eqref{toddeu} can be made explicit. For the first term, recalling \eqref{ki}, we immediately compute $$\pi_{\star} (\mathcal K_i^{\vee})=\mathbb C^{d_i+1}\otimes \mathcal O_{{\mathbb P}^{d_i}}(1)\implies \td(\pi_*(\mathcal{K}^\vee_i))= \left(\frac{h_i}{1-e^{-h_i}}\right)^{d_i+1},$$ 
	where $h_i$ is the hyperplane class on $\mathbb{P}^{d_i}$ (by abuse of notation also pulled back to $ \mathbb{P}^{d_1}\times\cdots\times \mathbb{P}^{d_N}$). The equivariant weights vanish for this term. (This is the Todd genus of the projective space, as it should.)
	
	Turning to the remaining terms, more generally, equation \eqref{ki} straightforwardly yields 
	\begin{align*}
		c(\pi_*(\mathcal{K}^\vee_i(a_j-a_{i})))&=(1+h_i)^{d_i+a_j-a_i+1}\\
		c(\pi_*(\mathcal{K}^{\vee}_i \otimes\mathcal{K}_j(a_{j}-a_{i})))&= (1+(h_i-h_j))^{d_i-d_j+a_j-a_i+1}.
	\end{align*}
In the equivariant cohomology, recalling that the above sheaves carry the weight $w_i-w_j$, we obtain 
	\begin{align*}
		c(\pi_*(\mathcal{K}^\vee_i(a_j-a_{i})))&=(1+(h_i+w_i\epsilon-w_j\epsilon))^{d_i+a_j-a_i+1}\\
		c(\pi_*(\mathcal{K}^{\vee}_i \otimes\mathcal{K}_j(a_{j}-a_{i})))&= (1+(h_i+w_i\epsilon-h_j-w_j\epsilon))^{d_i-d_j+a_j-a_i+1}.
	\end{align*}
Here, $\epsilon$ denotes the equivariant parameter. This implies the following expressions for the equivariant Todd genera
	\begin{align*}
		\td(\pi_*(\mathcal{K}^\vee_i(a_j-a_{i})))&= \bigg(\frac{h_i+w_i\epsilon-w_j\epsilon}{1-e^{-(h_i+w_i\epsilon-w_j\epsilon)}}\bigg)^{d_i+a_j-a_i+1}\\
		\td(\pi_*(\mathcal{K}^{\vee}_i \otimes\mathcal{K}_j(a_{j}-a_{i})))&= \bigg(\frac{h_i+w_i\epsilon-h_j-w_j\epsilon}{1-e^{-(h_i+w_i\epsilon-h_j-w_j\epsilon)}}\bigg)^{d_i-d_j+a_j-a_i+1}.
	\end{align*}
	Similarly we obtain the Euler classes
	\begin{align*}
		e_{\mathbb{C}^*}(\pi_*(\mathcal{K}^\vee_i(a_j-a_{i})))&= (h_i+w_i\epsilon-w_j\epsilon)^{d_i+a_j-a_i+1}\\
		e_{\mathbb{C}^*}(\pi_*(\mathcal{K}^{\vee}_i \otimes\mathcal{K}_j(a_{j}-a_{i})))&= (h_i+w_i\epsilon-h_j-w_j\epsilon)^{d_i-d_j+a_j-a_i+1}.
	\end{align*}
	\vskip.1in
\noindent {\bf Simplification.} All told, substituting the above expressions into \eqref{toddeu} and cancelling terms, we obtain 
	\begin{equation}\label{alltold}\frac{\td (\quot_d)}{e_{\mathbb{C}^*}(\textrm{N}_{\vec d})}\bigg|_{\fix_{\vec{d}}}=\prod_{i\in [N]}h_i^{d_i+1}\prod_{ i,j\in [N]}\bigg( \frac{z_i}{z_i-\alpha_j} \bigg)^{d_i+a_j-a_i+1} \prod_{i, j\in [N],\,i\ne j}\bigg(\frac{z_i-z_j}{z_i}\bigg)^{d_i-d_j+a_j-a_{i}+1}\end{equation} where we set for notational convenience $$z_i=e^{h_i+w_i\epsilon}, \quad \alpha_i=e^{w_i\epsilon}.$$ We rewrite this in a slightly more convenient form in terms of the polynomial $$R(z)=\prod_{j\in [N]}(z-\alpha_j).$$ Combining the $(i, j)$ and $(j, i)$-factors in the last product appearing in \eqref{alltold}, and judiciously accounting for the remaining terms, we eventually obtain	
	\begin{equation}\label{toddeufinal}\frac{\td (\quot_d)}{e_{\mathbb{C}^*}(\textrm{N}_{\vec d})}\bigg|_{\fix_{\vec{d}}}=\mathsf {u}\cdot \prod_{i\in [N]}\bigg(\frac{h_i}{R(z_i)} \bigg)^{d_i+1}z_i^{d+1} \bigg(\frac{R(z_i)}{\prod_{j\in [N]}(z_j-\alpha_i)}\bigg)^{a_i+\ell+1} 
		\cdot \prod_{i, j\in [N],\,\,i<j}(z_i-z_j)^2\end{equation} for the sign $$\mathsf u=(-1)^{(N-1)(d+\sum (a_i+\ell+1))+\binom{N}{2}}.$$
The integer $\ell$ included in the above expression will be useful later on. For now, the value of $\ell$ plays no role. Any $\ell$ will work since \[\prod_{i}\frac{R(z_i)}{\prod_{j}(z_j-\alpha_i)}=1. \] 

	\subsubsection{Chern classes} For the remaining term in \eqref{ab}, we record the following 
	\begin{lemma}\label{chern_classes}
		The equivariant restrictions of the Chern characters of the tautological bundles to the fixed loci are given by 
		\begin{align}
			\label{chl}\ch(\wedge_y L^{[d]} )\bigg|_{\fix_{\vec{d}}}&=\prod_{i} \bigg(\frac{z_i(\alpha_i+y)}{\alpha_i(z_i+y)}\bigg)^{a_i+\ell+1}\bigg(\frac{z_i+y}{z_i} \bigg)^{d_i}\\
			\label{chm}\ch\left((\wedge_{x}M^{[d]})^{\vee}\right)\bigg|_{\fix_{\vec{d}}}&=\prod_{i} \bigg(\frac{1+\alpha_ix}{1+z_ix}\bigg)^{a_i+m+1}(1+z_ix)^{d_i}
		\end{align}
		where $L$ and $M$ are line bundles of degree $\ell$ and $m$ respectively. 
	\end{lemma}
	\begin{proof}
		We only explain the first formula, the second assertion being entirely similar. We note that over $\textrm{F}_{\vec d}=\mathbb P^{d_1}\times \cdots \times \mathbb P^{d_N}$, the bundle $L^{[d]}$ splits as contributions coming from each factor $$L^{[d]}=\pi_{\star} (\mathcal Q\otimes p^{\star} L)=\bigoplus_{i\in [N]} \pi_{\star} (\mathcal T_i\otimes p^{\star} L(a_i)),$$ with each summand acted on with $\mathbb C^{\star}$-weight $-w_i$. In $K$-theory, we have by \eqref{ki} that $$\mathcal T_i=\mathcal O-\mathcal K_i=\mathcal O_{\mathbb P^1\times \mathbb P^{d_i}} - \mathcal O_{\mathbb P^1}(-d_i)\boxtimes \mathcal O_{\mathbb P^{d_i}}(-1).$$ This yields $$ \pi_{\star} (\mathcal T_i\otimes p^{\star} L(a_i))=\mathbb C^{a_i+\ell+1}\otimes \mathcal O_{\mathbb P^{d_i}}-\mathbb C^{a_i-d_i+\ell+1}\otimes \mathcal O_{\mathbb P^{d_i}}(-1).$$ The result follows immediately from here, using that $\wedge_y(V+W)=\wedge_y V\cdot \wedge_y W$ and accounting for all terms. \end{proof}
	\subsection{Proof of Theorem \ref{t1}} With the above ingredients in place, the key steps of the argument are as follows: 
	\begin{itemize}
	\item [(i)] after judiciously accounting for all localization terms, the fixed point contributions are summed using the Lagrange-B\"urmann formula;
	\item [(ii)] next, the answer is recast as a  quotient of suitable determinants. Schur polynomials evaluated at the roots of a certain algebraic equation arise at this step; 
	\item [(iii)] finally, an application of the Jacobi-Trudi formula to the Schur polynomials greatly simplifies the answer and gives the result. 
	\end{itemize}
	
	To begin, we substitute equations \eqref{toddeufinal} and \eqref{chl} into the localization expression \eqref{ab}. We obtain that $\chi(\quot_d,\wedge_y L^{[d]}) $ equals \begin{align}\label{weneed}\mathsf u
		\sum_{|\vec{d}|=d}\left[h_1^{d_1} \ldots h_N^{d_N}\right]&\bigg\{\prod_{i} \bigg(\frac{z_i(\alpha_i+y)}{\alpha_i(z_i+y)}\bigg)^{b_i}\bigg(\frac{z_i+y}{z_i}\bigg)^{d_i}\bigg(\frac{h_i}{R(z_i)} \bigg)^{d_i+1}z_i^{d+1}\bigg(\frac{R(z_i)}{\prod_{j}(z_j-\alpha_i)}\bigg)^{b_i}\\ & \cdot
	 \prod_{i<j}(z_i-z_j)^2\bigg\}\bigg{|}_{\epsilon=0}.\nonumber\end{align}
	Here, we wrote for simplicity $$b_i=a_i+\ell+1.$$ The brackets indicate extracting the coefficient of $h_1^{d_1} \ldots h_N^{d_N}$ in the relevant expression; this corresponds to integration over the  product of projective spaces $\textrm{F}_{\vec d}=\mathbb P^{d_1}\times \ldots \times \mathbb P^{d_N}.$ The equivariant parameter $\epsilon$ is set to $0$ at the end. 
	
	The rest of this section is dedicated to the explicit combinatorial manipulations (i)-(iii) which bring the above expression \eqref{weneed} into the form stated in Theorem \ref{t1}. We first apply the multivariable Lagrange-B\"urmann formula \cite{G}. The formulation we need in this case is as follows. Consider formal power series $\Phi_1(h_1), \ldots, \Phi_N(h_N)$ with $\Phi_i(0)\neq 0$, and consider a power series $\Psi(h_1, \ldots, h_N)$. We have 
\begin{equation}\label{formlb}\sum_{(d_1, \ldots, d_N)} q_1^{d_1} \cdots q_N^{d_N} \left[h_1^{d_1} \ldots h_N^{d_N}\right] \left(\Phi_1(h_1)^{d_1+1}\cdots \Phi_N^{d_N+1}(h_N) \cdot \Psi (h_1, \ldots, h_N)\right) =  \frac{\Psi}{J}\end{equation}  for the change of variables $$q_i=\frac{h_i}{\Phi_i(h_i)}$$ with Jacobian $$J= \frac{dq_1}{dh_1} \cdots \frac{dq_N}{dh_N}.$$ This formula will be used to derive equation \eqref{good} below. The intermediate calculations are straightforward; nonetheless, we record the details for completeness. 

Set $$\Phi_i(h_i)=\frac{h_i}{R(z_i)}\frac{z_i+y}{z_i},$$ and let 
$$\Psi=\mathsf u\,\prod_{i}\bigg(\frac{z_i(\alpha_i+y)}{\alpha_i(z_i+y)}\bigg)^{b_i}\bigg(\frac{z_i+y}{z_i}\bigg)^{-1}z_i^{d+1}\bigg(\frac{R(z_i)}{\prod_{j}(z_j-\alpha_i)}\bigg)^{b_i}\cdot
	 \prod_{i<j}(z_i-z_j)^2$$ be determined by the remaining terms in \eqref{weneed}. 
 \footnote{In the expression $\Psi$, we regard the exponent $d$ in the term $z_i^{d+1}$ as an independent parameter, foregoing for the moment the requirement that $d=d_1+\ldots+d_N$. The careful reader may wish to replace the term $z_i^{d+1}$ by a more general $z_i^{e+1}$ for $e\geq d$ in the proof below. This leads to \eqref{deref} written instead for the partition $\lambda_k=(e^N, k)$ or $(k-N, (e+1)^N)$. One then specializes back to the case of interest $e=d$ and continues by applying Lemma \ref{l4}.}  Due to the factor $z_i-\alpha_i$ in $R(z_i)$ which has a simple zero at $h_i=0$, we have $\Phi_i(0)\neq 0$. We apply \eqref{formlb} with $$q_1=\ldots=q_N=q.$$ Thus, letting $z_i$ be the root of the equation \begin{align*}
		q(z_i+y)=z_iR(z_i), \quad z_i\big{|}_{q=0}=\alpha_i,
	\end{align*}
and letting $h_i$ be determined by $z_i=\alpha_ie^{h_i}$, we have $q=\frac{h_i}{\Phi_i(h_i)}.$ It follows from \eqref{weneed} and \eqref{formlb} that $$\chi\left(\quot_d, \wedge_y L^{[d]}\right)=\left[q^d\right] \frac{\Psi}{J}\left(h_1(q), \ldots, h_N(q)\right).$$ Equivalently, $\chi\left(\quot_d, \wedge_y L^{[d]}\right)$ equals  
	\begin{align}\label{lag_bur}
\left[q^d\right] \mathsf {u} \prod_{i} \bigg(\frac{z_i(\alpha_i+y)}{\alpha_i(z_i+y)}\bigg)^{b_i}\bigg(\frac{z_i+y}{z_i}\bigg)^{-1}
		\frac{dh_i}{dq} z_i^{d+1} \bigg(\frac{R(z_i)}{\prod_{j}(z_j-\alpha_i)}\bigg)^{b_i} \cdot \prod_{i<j}(z_i-z_j)^2\bigg{|}_{\epsilon=0}.
	\end{align}

Consider the polynomial  
	\begin{equation}\label{eq:P(z)}
		P(z)=zR(z)-q(z+y).
	\end{equation}  Note that $P$ has degree $N+1$, so it admits $N+1$ roots, with $z_1, \ldots, z_N$ being $N$ of them. Let $z_{N+1}$ be the additional root of $P$ which satisfies $$z_{N+1}\big{|}_{q=0}=0.$$ We will greatly simplify \eqref{lag_bur} using the additional root $z_{N+1}$. 
	
	To this end, write $P(z)=(z-z_1)\cdots (z-z_{N+1})$. A simple calculation gives 
	\begin{align}\label{deriv}
		\frac{dq}{dh_i}=\frac{dq}{dz_i}\cdot \frac{dz_i}{dh_i}=z_i\frac{dq}{dz_i}= \frac{z_i}{z_i+y}P'(z_i).
	\end{align} Here, we used that $$q=\frac{z_i R(z_i)}{z_i+y}\implies \frac{dq}{dz_i}=\frac{R(z_i)}{z_i+y}+\frac{z_iR'(z_i)}{z_i+y}-\frac{z_iR(z_i)}{(z_i+y)^2}=\frac{P'(z_i)}{z_i+y},$$ where the definition of $P$ was necessary in the last equality. 
The terms in \eqref{lag_bur} with exponent $b_i$ further simplify since \begin{align}\label{simpl}
		\frac{z_i(\alpha_i+y)R(z_i)}{\alpha_i(z_i+y)\prod_{j=1}^{N}(z_j-\alpha_i)}=(-1)^N\frac{(z_{N+1}-\alpha_i)}{\alpha_i} \iff (\alpha_i+y) z_i R(z_i)=-(z_i+y)P(\alpha_i)
	\end{align}
	where the last identity holds true using \eqref{eq:P(z)} and the fact that $P(z_i)=0,$ $R(\alpha_i)=0$.	Substituting identities \eqref{deriv} and \eqref{simpl} into \eqref{lag_bur}, we obtain the expression
	\begin{equation}\label{good}
\left[ q^{d}\right] (-1)^{(N-1)d}\prod_{i=1}^{N}\bigg( \frac{\alpha_i-z_{N+1}}{\alpha_i}\bigg)^{b_i}\frac{z_i^{d+1}}{P'(z_i)}\cdot
		\prod_{1\le i\ne j\le N}(z_i-z_j)\bigg{|}_{\epsilon=0}.
	\end{equation}
	
	Having arrived at \eqref{good}, a new idea is needed to go further. Note that $$\prod_{i=1}^{N} P'(z_i)=\prod_{i=1}^{N}\prod_{\stackrel{j=1}{j\neq i}}^{N+1} (z_i-z_j)\implies 
	\prod_{i=1}^{N} \frac{1}{P'(z_i)}\cdot \prod_{1\le i\ne j\le N}(z_i-z_j)=\frac{\mathsf V_N}{\mathsf V_{N+1}}
	.$$ Here, we introduced the two Vandermonde determinants $$\mathsf V_N=\begin{vmatrix}
		z_1^{N-1}&\cdots& z_1&1\\
		z_2^{N-1}&\cdots &z_2&1\\
		\vdots&\cdots &\vdots &\vdots\\
		z_{N}^{N-1}&\cdots &z_{N}&1
	\end{vmatrix}, \quad \mathsf V_{N+1}=\begin{vmatrix}
		z_1^{N}&\cdots& z_1&1\\
		z_2^{N}&\cdots &z_2&1\\
		\vdots&\cdots &\vdots &\vdots\\
		z_{N+1}^{N}&\cdots &z_{N+1}&1
	\end{vmatrix}	.$$
We thus rewrite expression \eqref{good} in the form \begin{equation}\label{de}
	\left[q^d\right] \frac{(-1)^{(N-1)d}}{\mathsf V_{N+1}}\,\begin{vmatrix}
			z_1^{d+N}&z_1^{d+N-1}& \cdots& z_1^{d+1}\\
			z_2^{d+N}&z_2^{d+N-1}&\cdots &z_2^{d+1}\\
			\vdots& \vdots& \cdots &\vdots \\
			z_N^{d+N}&z_N^{d+N-1}&\cdots &z_N^{d+1}
		\end{vmatrix}\prod_{i=1}^{N}\bigg( \frac{\alpha_i-z_{N+1}}{\alpha_i}\bigg)^{b_i}\bigg{|}_{\epsilon=0}. 
	\end{equation}
	
	Next, recall that $[q^0]z_{N+1}=0$, hence we may add terms which are multiples of $z_{N+1}^{d+1}$ without changing the coefficient of $q^d$. Using this observation, we enlarge the determinant in the numerator by adding one more row and column. The answer is recast as the quotient of two determinants of size $(N+1)\times (N+1)$: \begin{align}\label{de2}
		\left[\epsilon^0 q^d\right]\frac{(-1)^{(N-1)d}}{\mathsf {V}_{N+1}}\,
		\begin{vmatrix}
			z_1^{d+N}& z_1^{d+N-1}&\cdots & z_1^{d+1}& \prod_{i=1}^{N}(\frac{\alpha_i-z_1}{\alpha_i})^{b_i}\\
			z_2^{d+N}& z_2^{d+N-1}&\cdots & z_2^{d+1}& \prod_{i=1}^{N}(\frac{\alpha_i-z_2}{\alpha_i})^{b_i}\\
			\vdots& \vdots &\cdots & \vdots& \vdots
			\\
			z_{N+1}^{d+N}& z_{N+1}^{d+N-1}&\cdots & z_{N+1}^{d+1}& \prod_{i=1}^{N}(\frac{\alpha_i-z_{N+1}}{\alpha_i})^{b_i}
		\end{vmatrix}.
	\end{align}
This does not change expression \eqref{de}. Indeed, expanding along the last row, the first entries do not contribute by the above reasoning, while the rightmost corner contribution matches \eqref{de}. The assumption $\deg L+a_i+1\geq 0$ made in the beginning of this section is also used here. This condition rewrites as $b_i\geq 0$, so the terms we added on the last column do not contribute poles at $q=0$.  

Expression \eqref{de2} is symmetric in the roots $z_1, \ldots, z_{N+1}$ of $P(z)$. The answer can be rewritten in terms of the elementary symmetric functions in the $z_i$'s which depend polynomially on the $\alpha_i$'s. Thus \eqref{de2} is a rational fraction in the $\alpha$'s, with denominator $\prod_{i=1}^{N} \alpha_i^{b_i}$ (coming from the last column). Since we are interested in the coefficient of $\epsilon^0$, since the elementary symmetric functions depend continuously on the $\alpha$'s, we may substitute $\alpha_i=1$, noting that there are no poles in \eqref{de2} at these values. 

After the substitution, the $z_i$'s solve $z(z-1)^N-q(z+y)=0$. Furthermore, since $$\sum_{i=1}^{N} b_i=\chi(E\otimes L):=\chi,$$ the entries in the last column of the determinant \eqref{de2} become $$(1-z_i)^{\chi}=\sum_{k\geq 0} (-1)^{k}\binom{\chi}{k} z_i^{k}.$$ 

Expanding the determinant along the last column yields sums over Schur polynomials. Specifically, we obtain \begin{equation}\label{deref}\left[q^d\right] \sum_{k\geq 0} (-1)^{(N-1)d+k}\binom{\chi}{k} s_{\lambda_k}(z_1, \ldots, z_{N+1}).\end{equation} Here, we set $$\lambda_k=(d^N,k)=(d, \ldots, d, k),$$ and $s_{\lambda_k}(z_1,\dots z_{N+1})$ denotes the corresponding Schur polynomial, when $k\leq d$. The terms for $d<k\leq d+N$ have vanishing contribution due to repeating columns in the determinant.
To account for the ordering of the exponents, the shape of the partition changes when $k>d+N$. In all cases, we find $$\lambda_k=\begin{cases} (d^N, k) & \text{if } k\leq d\\ (k-N, (d+1)^N)&\text{if } k>d+N\end{cases}.$$ The lemma below identifies the coefficient of $q^d$ in $s_{\lambda_k}(z_1, \ldots, z_{N+1})$. We obtain 
	\begin{align*}
		\chi(\quot_d, \wedge_yL^{[d]})= \sum_{k=0}^{d}\binom{\chi}{k} y^{k}= \left[q^d\right](1+qy)^{\chi}(1-q)^{-1}.
	\end{align*}
This completes the proof of Theorem \ref{t1} in genus $0$ under the assumption $b_i\geq 0$ for all $1\leq i\leq N$. 
	
	\begin{lemma}\label{l4}
		We have $$\left[q^d\right]s_{\lambda_k}(z_1,\dots,z_{N+1})=\begin{cases}
			(-1)^{d(N-1)}(-y)^k & \text{if }k\le d\\
			0&\text{if }k>d+N
		\end{cases}.$$
	\end{lemma}
	\begin{proof}
		Since the $z_i$'s are the roots of the polynomial $P(z)=z(z-1)^N-q(z+y)$, the elementary symmetric functions in $z_1, \ldots, z_{N+1}$ are
		\begin{align*}
			e_j=\begin{cases}
				\binom{N}{j} &\text{if }j\ne N,N+1\\
				1+(-1)^{N-1}q &\text{if } j=N\\
				(-1)^Nqy &\text{if }j=N+1.\\
			\end{cases}
		\end{align*}
		Assume $k\leq d$ so that $\lambda_k=(d^N, k)$. The Jacobi-Trudi formula expresses the Schur polynomial as a $d\times d$ determinant in the elementary symmetric functions. The entries are dictated by the conjugate partition $\lambda_k'=((N+1)^k, N^{d-k}),$ so that
		\begin{align}\label{jt}
				s_{\lambda_k}= \left|\begin{array}{ccccccccccc}
				e_{N+1}& 0&0&\cdots &0& \rvline&\quad0 &\cdots &0&0\\
				e_{N}&e_{N+1}&0&\cdots &0&\rvline&\quad 0&\cdots  &0&0\\
				e_{N-1}&e_{N}&e_{N+1}&\cdots & 0&\rvline&\quad 0 &\cdots &0&0\\
				\vdots & \vdots &\vdots&\cdots&\vdots &\rvline&\,\,\,\,\;\;\vdots&\cdots&\vdots&\vdots&\\
				e_{N-k+2}&e_{N-k+3}&e_{N-k+4} &\cdots & e_{N+1}& \rvline & \quad 0&\cdots &0&0\\
				&&&&&\rvline&\quad\quad\,\,\,&&&\\
				\hline
				e_{N-k}&e_{N-k+1}& e_{N-k+2} &\cdots &e_{N-1}&\rvline& \quad e_N & \cdots &0&0\\
				\vdots&\vdots&\vdots&\cdots&\vdots&\rvline& \vdots&\cdots&\vdots &\vdots\\
				e_{N-d+2}&e_{N-d+3}&e_{N-d+4}&\cdots& e_{N-d+k+1}&\rvline&\quad e_{N-d+k+2}&\cdots&e_{N} &e_{N+1}\\
				e_{N-d+1}&e_{N-d+2}&e_{N-d+3}&\cdots&e_{N-d+k}&\rvline& \quad e_{N-d+k+1}&\cdots&e_{N-1} &e_N\\
			\end{array}\right|.
		\end{align}
		Each of the $e_j$'s is at most linear in $q$. Since the determinant has size $d$, extracting the $q^d$ coefficient is immediate. In fact, we can replace the $e_j$'s by their linear terms in $q$; these are zero unless $j=N$ or $j=N+1$. We obtain that 
		\begin{align*}
			\left[q^d\right] s_{\lambda_k}= \left[q^d\right]\left|\begin{array}{cccccc}\begin{matrix}
				e_{N+1}& 0&0&\cdots &0 &0\\
				e_{N}&e_{N+1}&0&\cdots &0 &0\\
				\vdots &\vdots &\vdots &\cdots &\vdots &\vdots\\
				0& 0 & 0 & \cdots & e_{N}&e_{N+1}\end{matrix}
				& \rvline & \bigzero\\
				\hline
				\bigzero & \rvline &
				\begin{matrix}
				e_{N}& e_{N+1}&0&\cdots &0 &0\\
				0&e_{N}&e_{N+1}&\cdots &0 &0\\
				\vdots &\vdots &\vdots &\cdots &\vdots &\vdots\\
				0& 0 & 0 & \cdots & 0&e_{N}\end{matrix}
			\end{array}\right|.
		\end{align*}
Thus, 		 
		\begin{align*}
			[q^d]s_{\lambda_k}=\left[q^d\right]e_{N+1}^{k} e_{N}^{d-k}=(-1)^{kN+(d-k)(N-1)}y^k.
		\end{align*}
		The case $k>d+N$ changes the conjugate partition $\lambda'_k$, but the reasoning is identical. 
		\end{proof}
	\subsection{Proof of Theorem \ref{t2}}
	A similar but slightly more involved argument yields Theorem \ref{t2} in genus $0$ when $b_i\geq 0$ for all $1\leq i\leq N$.  
	Specifically, we prove that 
\begin{equation}\label{t2eq} \chi\left(\quot_d, \wedge_yL^{[d]}\otimes_{p=1}^{r}\left(\wedge_{x_p} M_p^{[d]}\right)^{\vee}\right)= \left[q^{d}\right](1-q)^{-1}(1+qy)^{\chi(E\otimes L)}\prod_{p=1}^{r}(1-x_p y q)^{-\chi(L\otimes M_p^\vee)}.\end{equation} We indicate some of the steps. 
	
	Just as in Theorem \ref{t1}, we begin by applying Hirzebruch-Riemann-Roch followed by Atiyah-Bott localization:
	\begin{equation}\label{gen}
		\chi\left(\quot_d, \wedge_yL^{[d]}\otimes_{p=1}^{r}\left(\wedge_{x_p} M_p^{[d]}\right)^{\vee}\right) = \sum_{\vec{d}}\int_{\fix_{\vec{d}}}
		\ch(\wedge_y L^{[d]})\prod_{p=1}^{r}\ch\left(\left(\wedge_{x_p} M_p^{[d]}\right)^\vee\right)\frac{\td(\quot_d)}{e_{\mathbb{C}^*}(\textrm{N}_{\vec d})}\bigg|_{\fix_{\vec{d}}}.
	\end{equation}
	All terms that appear here have been computed in the previous subsections. Using \eqref{toddeufinal}, \eqref{chl} and \eqref{chm}, we rewrite \eqref{gen} as	
	\begin{align*}
		\mathsf {u} \sum_{|\vec{d}|=d}\left[h_1^{d_1}\cdots h_N^{d_N}\right]&\,\bigg\{\,\prod_{i=1}^{N}\,\,\bigg( \bigg(\frac{z_i(\alpha_i+y)}{\alpha_i(z_i+y)}\bigg)^{b_i}\bigg(\frac{z_i+y}{z_i}\bigg)^{d_i}\prod_{p=1}^{r}\bigg(\frac{1+\alpha_ix_p}{1+z_ix_p}\bigg)^{a_i+m_p+1}(1+z_ix_p)^{d_i}\bigg(\frac{h_i}{R(z_i)} \bigg)^{d_i+1}\\
		&z_i^{d+1} \bigg(\frac{R(z_i)}{\prod_{j=1}^{N}(z_j-\alpha_i)}\bigg)^{b_i}\bigg) \cdot \prod_{1\leq i<j\leq N}(z_i-z_j)^2\bigg\}\,\bigg{|}_{\epsilon=0}
	\end{align*}
	where $b_i=a_i+\ell+1$ and $\mathsf u=(-1)^{(N-1)(d+\sum b_i)+\binom{N}{2}}$. Here, we set $m_p=\deg M_p$. 
	
	Next, the Lagrange-B\"urmann formula with the change of variables 
	\begin{align}\label{chvar}
		q(z_i+y)\prod_{p=1}^{r}(1+z_ix_p)=z_iR(z_i)
	\end{align}
	turns \eqref{gen} into the following unwieldy expression
	\begin{align*}\label{eq:lag_bur}
		\left[q^d\right] \,\mathsf u\,\, \prod_{i=1}^N&\left[\bigg(\frac{z_i(\alpha_i+y)}{\alpha_i(z_i+y)}\bigg)^{b_i} \prod_{p=1}^{r}\bigg(\frac{1+\alpha_ix_p}{1+z_ix_p}\bigg)^{a_i+m_p+1}\bigg(\frac{z_i+y}{z_i}\prod_{p=1}^{r}(1+z_ix_p)\bigg)^{-1}
		\frac{dh_i}{dq} z_i^{d+1} \bigg(\frac{R(z_i)}{\prod_{j}(z_j-\alpha_i)}\bigg)^{b_i}\right] \\ &\cdot\prod_{1\leq i<j\leq N}(z_i-z_j)^2\bigg{|}_{\epsilon=0}.
	\end{align*}
However, there are further simplifications. To this end, we define the polynomial \[P(z)=zR(z)-q(z+y)\prod_{p=1}^{r}(1+zx_p). \] Since $r\leq N-1$, the degree of $P$ is $N+1$, so there is an additional root $z_{N+1}$ for $P$. Following the same steps that led to \eqref{good}, we simplify the above expression to 
	\begin{equation}\label{expr}
		\left[q^d\right](-1)^{(N-1)d} f(z_{N+1})\,\prod_{i=1}^{N}\,\frac{z_i^{d+1}}{P'(z_i)}\prod_{1\leq i\neq j\leq  N}(z_i-z_j)\bigg{|}_{\epsilon=0}
	\end{equation}
	where $$f(z)=\prod_{p=1}^{r}(1+zx_p)^{m_p-\ell}\prod_{i=1}^{N}\left(\frac{\alpha_i-z}{\alpha_i}\right)^{b_i}.$$
We record the details of the simplification in the lemma below; the reader can also skip directly to \eqref{degen}. 
	\begin{lemma} We have 
\begin{align}\label{eq:der}
		\bigg(\frac{z_i+y}{z_i}\prod_{p=1}^{r}(1+z_ix_p)\bigg)\frac{dq}{dh_i} =P'(z_i)
	\end{align}
and \begin{equation}\label{eq:simpli_2}
		\prod_{i=1}^{N}\left(\bigg(\frac{z_i(\alpha_i+y)}{\alpha_i(z_i+y)}\, \frac{R(z_i)}{\prod_{j=1}^{N}(z_j-\alpha_i)}\bigg)^{b_i}\prod_{p=1}^{r}\bigg(\frac{1+\alpha_ix_p}{1+z_ix_p}\bigg)^{a_i+m_p+1}\right)=(-1)^{(N-1)\sum b_i} f(z_{N+1}).
	\end{equation}

\end{lemma} 
\proof Equation \eqref{eq:der} follows by differentiating the expression for $q$ given in \eqref{chvar}. For \eqref{eq:simpli_2}, recall $b_i=a_i+\ell+1$, and use the following identities
	\begin{align*}
		z_iR(z_i)&=q(z_i+y)\prod_{p=1}^{r}(1+z_ix_p),\\
		\prod_{j=1}^{N}(z_j-\alpha_i)&=(-1)^{N+1}\frac{P(\alpha_i)}{z_{N+1}-\alpha_i}=(-1)^Nq\frac{(\alpha_i+y)\prod_{p=1}^{r}(1+\alpha_ix_p)}{z_{N+1}-\alpha_i}.
	\end{align*}
In the last line we used the definition of $P$ and the fact that $R(\alpha_i)=0.$ Then \eqref{eq:simpli_2} becomes
	\begin{equation*}
		\prod_{i=1}^{N}\left(\bigg((-1)^N\frac{(z_{N+1}-\alpha_i)}{\alpha_i}\bigg)^{b_i}\prod_{p=1}^r\bigg(\frac{1+\alpha_ix_p}{1+z_ix_p}\bigg)^{m_p-\ell}\right) .
	\end{equation*}
Finally, recalling that $\alpha_i$ and $z_i$ are roots of $R$ and $P$, for each fixed $p$ we have 
	\begin{align*}
		\prod_{i=1}^{N}\frac{1+\alpha_ix_p}{1+z_ix_p}
		= \frac{R\left(-1/x_p\right)}{P\left(-1/x_p\right)}\left(-\frac{1}{x_p}-z_{N+1}\right)
		= (1+z_{N+1}x_p). 
	\end{align*} 
In the last equality, we used again the definition of $P$ in terms of $R$. The lemma follows from here. \qed
\vskip.1in
	Having arrived at \eqref{expr}, by the same reasoning as in \eqref{de} we rewrite the answer as the quotient of two determinants \begin{equation}\label{degen}
		\left[\epsilon^0 q^{d}\right] \frac{(-1)^{(N-1)d}}{\det(z_i^{N-j+1})}\,\begin{vmatrix}
			z_1^{d+N}&z_1^{d+N-1}&\cdots& z_1^{d+1}\\
			z_2^{d+N}&z_2^{d+N-1}&\cdots &z_2^{d+1}\\
			\vdots&\vdots &\cdots & \vdots\\
			z_N^{d+N}&z_N^{d+N-1}&\cdots &z_N^{d+1}
		\end{vmatrix}f(z_{N+1}).
	\end{equation}
	The denominator is the Vandermonde determinant of size $(N+1) \times (N+1)$, while the numerator has size $N \times N$. Using the previous arguments, in particular that $[q^0]z_{N+1}=0$, we enlarge the determinant appearing in the numerator of \eqref{degen} by adding one row and one column:
	\begin{align}\label{det2}
		[\epsilon^0 q^d]\frac{(-1)^{(N-1)d}}{\det(z_i^{N-j+1})}\,
		\begin{vmatrix}
			z_1^{d+N}& z_1^{d+N-1}&\cdots & z_1^{d+1}&  f(z_{1})\\
			z_2^{d+N}& z_2^{d+N-1}&\cdots & z_2^{d+1}& f(z_{2})\\
			\vdots& \vdots &\cdots & \vdots& \vdots
			\\
			z_{N+1}^{d+N}& z_{N+1}^{d+N-1}&\cdots & z_{N+1}^{d+1}& f(z_{N+1})
		\end{vmatrix}.
	\end{align}
Since \eqref{det2} is symmetric in $z_i's$, it can be written as a rational function in the $\alpha_i$'s whose denominator equals $\prod_{i=1}^{N} \alpha_i^{b_i}$ coming from the denominator of $f$. The substitution $\alpha_i=1$ therefore makes sense. After this substitution, the last column can be rewritten in terms of 
$$f(t)|_{\alpha_i=1}=\prod_{p=1}^{r}(1+x_p t)^{m_p-\ell}\cdot \left(1-t\right)^{\chi}$$ for the values $t=z_1, z_2, \ldots, z_{N+1}$. Here $\chi=\chi(E\otimes L).$ We expand $f(t)$ into powers $t^k$, and then we expand the determinant \eqref{det2} along the last column yielding 
	\begin{equation}\label{det3}
		(-1)^{(N-1)d}\sum_{k\geq 0}\left[t^k\right]f(t)\cdot\left[q^{d}\right]s_{\lambda_k} (z_1, \ldots, z_{N+1}),
	\end{equation}
for the partition $$\lambda_k=\begin{cases} (d^N, k) & \text{if }k\leq d\\ (k-N, (d+1)^N) &\text{if }k>d+N.\end{cases}$$ 
	By Lemma \ref{lem:schur_cal} below, for $k\leq d$ we have 
	\[\left[q^{d}\right]s_{\lambda_k} (z_1, \ldots, z_{N+1})= (-1)^{(N-1)d}(-y)^k\left[t^{d-k}\right]\frac{1}{(1-t)(1-x_1yt)\cdots (1-x_ryt)}.  \]
Substituting the last formula into \eqref{det3}, we obtain that \begin{align*}
	\chi&\left(\quot_d, \wedge_yL^{[d]}\otimes_{p=1}^{r}\big(\wedge_{x_p} M_p^{[d]} \big)^{\vee}\right)\\
			&=\sum_{k=0}^{d}\left [t^k\right]\left((1-t)^{\chi(E\otimes L)}\prod_{p=1}^{r}(1+x_pt)^{m_p-\ell}\right)\cdot (-y)^k\left[t^{d-k}\right]\frac{1}{(1-t)\prod_{p=1}^{r}(1-x_pyt)}\\
		&= 	\sum_{k=0}^d \left[t^k\right]\left((1+yt)^{\chi(E\otimes L)}\prod_{p=1}^{r}(1-x_pyt)^{m_p-\ell}\right)\cdot\left [t^{d-k}\right]\frac{1}{(1-t)\prod_{p=1}^{r}(1-x_pyt)}
		\\&=\left[t^{d}\right]\frac{(1+yt)^{\chi(E\otimes L)}}{(1-t)\prod_{p=1}^{r}(1-x_p y t)^{\chi(L\otimes M_p^\vee)}}.
	\end{align*}
This completes the proof of \eqref{t2eq} and of Theorem \ref{t2} in genus $0$ when $b_i\geq 0$ for all $i$. 
	
	\subsection{Schur polynomials}
	Let $z_1,\dots z_{N+1}$ be $N+1$ roots of $$P(z)=z(z-1)^N-q(z+y)(1+zx_1)\cdots (1+zx_r),$$
	where $0\le r\le N-1$. We show
	
	\begin{lemma}\label{lem:schur_cal}
		For the partition $\lambda_k$ above, and $k\leq d$, we have \[\left[q^d\right]s_{\lambda_k}(z_1,\dots,z_{N+1})=
		(-1)^{(N-1)d}(-y)^k\left[t^{d-k}\right]\frac{1}{(1-t)(1-x_1yt)\cdots (1-x_ryt)}.
		\] If $k>d+N$, the coefficient vanishes.  
	\end{lemma}
	\begin{proof} The proof is similar to that of Lemma \ref{l4}. Assume $k\leq d$, the other case being analogous. 
		Since the $z_i$'s are the roots of the polynomial $P(z)$, the elementary symmetric functions are
		\begin{align*}
			e_j= \binom{N}{j}+(-1)^{j-1}q\left[z^{N+1-j}\right](y+z)(1+zx_1)\cdots (1+zx_r).
		\end{align*}
		
		We examine again the Jacobi-Trudi determinant \eqref{jt} \begin{align*}
			s_{\lambda_k}= \left|\begin{array}{cccccccccc}
				e_{N+1}& 0&\cdots &0& \rvline&\quad0 &\cdots &0&0\\
				e_{N}&e_{N+1}&\cdots &0&\rvline&\quad 0&\cdots  &0&0\\
				e_{N-1}&e_{N}&\cdots & 0&\rvline&\quad 0 &\cdots &0&0\\
				\vdots & \vdots &\cdots&\vdots &\rvline&\quad\vdots&\cdots&\vdots&\vdots&\\
				e_{N-k+2}&e_{N-k+3}&\cdots & e_{N+1}& \rvline & \quad 0&\cdots &0&0\\
				&&&&\rvline&&&&&\\
				\hline
				e_{N-k}&e_{N-k+1} &\cdots &e_{N-1}&\rvline& \quad e_N & \cdots &0&0\\
				\vdots&\vdots&\cdots&\vdots&\rvline& \vdots&\cdots&\vdots &\vdots\\
				e_{N-d+2}&e_{N-d+3}&\cdots& e_{N-d+k+1}&\rvline&\quad e_{N-d+k+2}&\cdots&e_{N} &e_{N+1}\\
				e_{N-d+1}&e_{N-d+2}&\cdots&e_{N-d+k}&\rvline& \quad e_{N-d+k+1}&\cdots&e_{N-1} &e_N\\
			\end{array}\right|.
		\end{align*}
The $e_j$'s are at most linear in $q$. To find the coefficient of $q^d$ in the above $d\times d$ determinant, we may thus replace $e_j$ with the coefficient of the linear term in $q$. Thus, we may take \begin{equation}\label{ejnew}e_j = (-1)^{j-1}\left[z^{N+1-j}\right](y+z)(1+zx_1)\cdots (1+zx_r).\end{equation} In particular $e_{N+1}=(-1)^N y.$
		Furthermore, note that the first $k\times k$ block of the determinant is lower triangular, hence 
		\begin{equation*}
			\left[q^d\right]s_{\lambda_k}=e_{N+1}^{k}\cdot T_{d-k} = (-1)^{Nk} y^k \cdot T_{d-k}
		\end{equation*}
		where $T_m$ is the $m\times m$ determinant
		\begin{equation*}
			T_m=\begin{vmatrix}
				e_N&e_{N+1}&0& \cdots &0&0\\
				e_{N-1}&e_{N}&e_{N+1}& \cdots &0&0\\
				e_{N-2}&e_{N-1}&e_{N}& \cdots &0&0\\
				\vdots & \vdots &\vdots&\cdots &\vdots
				\\
				e_{N-m+2}&e_{N-m+3}&e_{N-m+2}& \cdots &e_{N}&e_{N+1}
				\\e_{N-m+1}&e_{N-m+2}&e_{N-m+1}& \cdots &e_{N-1}&e_{N}
			\end{vmatrix}.
		\end{equation*}
		The argument is completed using the Lemma below.
	\end{proof}
	\begin{lemma}
		Assume $e_1, \ldots, e_{N+1}$ are given by \eqref{ejnew}. For any $m\ge 0$, we have 
		\begin{equation}
			T_m=(-1)^{(N-1)m}\left[t^{m}\right]\frac{1}{(1-t)(1-x_1yt)\cdots (1-x_ryt)}.
		\end{equation}
	\end{lemma}
	\begin{proof} We set $T_0=1$ and $T_\ell=0$ for $\ell<0$. 
		By expanding the determinant $T_m$ along the first column and then successively along the rows, we obtain the recursion 
		\begin{equation*}
			T_m=\sum_{j=0}^{r}(-1)^je_{N+1}^je_{N-j} T_{m-j-1}\quad \text{ for all } m>0. 
		\end{equation*} Note that by \eqref{ejnew}, for degree reasons we have $e_{N-j}=0$ if $j>r$. This explains the upper bound of the index $j$ in the sum. 
		Forming the generating series $$T=\sum_{m=0}^{\infty} T_m t^m,$$ the above recursion immediately yields $$T=\bigg(1-\sum_{j=0}^{r}(-1)^je_{N+1}^je_{N-j} t^{j+1}\bigg)^{-1}.$$
		Substituting the values of $e_j$ from \eqref{ejnew}, we obtain for all $0\leq j\leq r$ that 
		\begin{align*}
			(-1)^{j+1}e_{N+1}^je_{N-j}&=(-1)^{N(j+1)}y^{j+1}\left[z^{j+1}\right]\left(\left(1+\frac{z}{y}\right)(1+zx_1)\cdots (1+zx_r)\right)\\
			&= \left[t^{j+1}\right]\left((1-(-1)^{N-1}t)(1-(-1)^{N-1}x_1yt)\cdots (1-(-1)^{N-1}x_ryt)\right),
		\end{align*} 
		where the substitution $z=(-1)^{N} yt$ was carried out in the last step. 
		Therefore
		\begin{align*}
		T=\bigg(1-\sum_{j=0}^{r}(-1)^je_{N+1}^je_{N-j} t^{j+1}\bigg)^{-1}= \frac{1}{(1-(-1)^{N-1}t)\prod_{p=1}^{r}(1-(-1)^{N-1}x_pyt)}.
		\end{align*}
		Taking the coefficient of $t^m$ gives the required identity.
	\end{proof}
	
		\subsection{Arbitrary genus}\label{arb} Relying on the ideas of \cite{EGL}, we show how the calculations for $C=\mathbb{P}^1$ imply Theorems \ref{t1} and \ref{t2} for arbitrary genus. We explain this for Theorem \ref{t1}, the case of Theorem \ref{t2}  being entirely similar. The argument is also noted and used in \cite{OP} over surfaces for punctual quotients of trivial bundles, and extended to quotients of arbitrary vector bundles in \cite{Sta2}. The case of curves is analogous, but we record the details for the benefit of the readers who seek a self-contained account. 
\vskip.1in

{\it Step1.}  The first goal is to show that the Euler characteristic \begin{equation}\label{chily}\chi(\quot_d(E), \wedge_y L^{[d]})\end{equation} is a polynomial (that may depend on $N$) in $\deg E$, $\deg L$ and $\chi(\mathcal O_C)$ (with coefficients in $\mathbb Q[y]$), for all smooth projective possibly disconnected curves $C$. In fact, the statement holds for all tautological integrals of the form \begin{equation}\label{tautological} \int_{\quot_d(E)} \mathsf P\end{equation} where $\mathsf P$ is a polynomial in the Chern classes of the tangent bundle of $\quot_d(E)$ and $L^{[d]}.$ 

We first analyze the case of split vector bundles $$E=\bigoplus_{i=1}^{N} F_i, \quad \text{rk }F_i=1.$$ For such a vector bundle, we can use the action of $\mathbb C^{\star}$ on the summands of $E$ to evaluate \eqref{tautological}, just as we have done for genus $0$ above. This way, we are led to considering integrals of the form \begin{equation}\label{locpro}\int_{C^{[d_1]}\times \cdots \times C^{[d_N]}} \mathsf Q\end{equation} where $\mathsf Q$ is a polynomial involving the Chern classes of $$\pi_{\star}\left(\mathcal K_i\otimes M\right),\quad \pi_{\star} \left(\mathcal K_i^{\vee}\otimes M\right), \quad \pi_{\star}\left(\mathcal K_i^{\vee}\otimes \mathcal K_j\otimes M\right)$$ for various $M\to C$, including $M=F_i^{\vee} \otimes F_j$ or $M=L\otimes F_i$. We can evaluate these with the aid of Grothendieck-Riemann-Roch. The integrals \eqref{locpro} can be pulled back via the finite map $$p: C^d\to C^{[d_1]}\times \cdots \times C^{[d_N]}.$$ The pullbacks of $\mathcal K_i^{\vee}$ over $C^d\times C$ correspond to sums of diagonals $\Delta_{\bullet, d+1}$, and thus $\eqref{locpro}$ takes the form $$\frac{d_1!\cdots d_N!}{d!}\int_{C^d} \widetilde{\mathsf Q}$$ where $\widetilde{\mathsf Q}$ is a universal expression in the diagonals and classes from $C$. In general, monomials in diagonals and classes from $C$ can be evaluated explicitly using that for $\Delta\hookrightarrow C\times C$ we have $$\Delta^{2}=2\chi(\mathcal O_C),\quad \Delta\cdot M=\deg M,$$ for all smooth projective possibly disconnected curves $C$, $M\to C$. Therefore \eqref{tautological} is a polynomial in $\deg F_i$, $\deg L$ and $\chi(\mathcal O_C)$. 
	
{\it Step 2.} We next argue that the above polynomial only depends on $\deg E=\sum_i \deg F_i$, $\deg L$ and $\chi(\mathcal O_C)$. This requires additional considerations. We write $$x_i=\deg F_i, \quad y=\deg L, \quad z=\chi(\mathcal O_C),$$ and $\mathsf R(x_1, \ldots, x_N, y, z)$ for the universal polynomial found above. The polynomial $\mathsf R$ is certainly symmetric in $x_1, \ldots, x_N$. 
	
	We claim that if $x_i$ are sufficiently large, $\mathsf R(x_1, \ldots, x_N, y, z)$ is in fact a polynomial in $\sum_{i=1}^{N} x_i.$ Indeed, for large degrees, the line bundles $F_i$ are globally generated (over connected curves $C$). Thus we can write $E$ as a quotient $$0\to K\to W\to E\to 0,$$ where $W$ is a trivial bundle (whose rank depends on $\deg E$). By \cite [Theorem 5]{Sta1}, modified from the original setting of surfaces to the case of curves, there is an embedding \begin{equation}\label{emb}\quot_d(E)\hookrightarrow \quot_d(W)\end{equation} cut out by a canonical section of the bundle $\left (K^{\vee}\right)^{[d]}$. With this observation, the integral \eqref{tautological} rewrites as \begin{equation}\label{tautological2}\int_{\quot_d(E)} \mathsf P=\int_{\quot_d(W)} \widetilde {\mathsf P}\end{equation} where $\mathsf P$ is a polynomial in the Chern classes of the tangent bundle of 
$\quot_d(W)$ and the tautological bundles $\left(K^{\vee}\right)^{[d]}$ and $L^{[d]}$. Applying the localization argument in {\it Step 1} once again, this time to $\quot_d(W),$ we see that \eqref{tautological2} only depends on $$\deg K^{\vee}=\deg E, \quad \deg L, \quad \chi(\mathcal O_C).$$ Thus, the polynomial $\mathsf R(x_1, \ldots, x_N, y, z)$ is a function of $\sum_{i=1}^{N} x_i, y, z,$ when $x_i$ are large. Hence $$\mathsf R(x_1, \ldots, x_N, y, z)=\mathsf S(x_1+\ldots+x_N, y, z),$$ for a new universal polynomial $\mathsf S$. This proves the statement we need about \eqref{tautological} when the bundle $E$ splits. 
	
{\it Step 3.} The general case follows from the following observation. Assume $E$ sits in an extension $$0\to E_1\to E\to E_2\to 0.$$ Considering the universal extension $$0\to p^{\star} E_1\to \mathcal E\to p^{\star} E_2\to 0$$ over $p:C\times \text{Ext}^{1} (E_2, E_1)\to C$, and constructing the relative Quot scheme $\quot_d(\mathcal E)$ over the extension space, we see that  $$\int_{\quot_d(E)} \mathsf P= \int_{\quot_d(E_1\oplus E_2)} \mathsf P.$$ To reduce to the case of split $E$, consider $M$ a line bundle such that $$0\to M\to E\to F\to 0$$ is exact and $F$ is a vector bundle of smaller rank. By the above observation we can replace $E$ by $M\oplus F$, and then continue inductively.

	\vskip.1in
		
{\it Step 4.} We return to the series appearing in Theorem \ref{t1}, namely $$\mathsf{Z}_{C, L, E}=\sum_{d=0}^{\infty} q^d \chi \left(\quot_{d}(E), \wedge_y L^{[d]}\right).$$ 
Consider a disconnected curve $C=C_1\sqcup C_2$, $E=E_1\sqcup E_2$ and $L=L_1\sqcup L_2$. We compare the Quot schemes of $C, C_1, C_2$ and the tautological bundles over them: $$\mathsf {Quot}_d(E)=\bigsqcup_{d_1+d_2=d}\mathsf {Quot}_{d_1}(E_1)\times \mathsf {Quot}_{d_2}(E_2), \quad L^{[d]}=\bigsqcup_{d_1+d_2=d} L_1^{[d_1]}\boxtimes L_2^{[d_2]}.$$ This implies \begin{equation}\label{e3}{\mathsf Z}_{C, L, E}={\mathsf Z}_{C_1, L_1, E_1}\cdot {\mathsf Z}_{C_2, L_2, E_2}.\end{equation} By the arguments of \cite[Theorem 4.2]{EGL}, the factorization \eqref{e3} shows that \begin{equation}\label{factor2}\mathsf Z_{C, L, E}=\mathsf A^{\chi(C, \mathcal O_C)} \cdot {\mathsf B}^{\deg L}\cdot \mathsf C^{\deg E},\end{equation} for universal series $\mathsf A, \mathsf B, \mathsf C \in \mathbb Q(y)[[q]]$ that depend only on $N$. We specialize to $(C, L)=(\mathbb P^1, \mathcal O_{\mathbb P^1}(\ell))$ with $\ell$ sufficiently large, and $E=\mathcal O(a_1)\oplus \ldots \oplus \mathcal O(a_N)$. Comparing \eqref{answ} and \eqref{factor2}, we obtain  $${\mathsf A}=(1-q)^{-1}\cdot (1+qy)^N, \quad {\mathsf B}=(1+qy)^N, \quad {\mathsf C}=1+qy.$$ Substituting these expressions back into \eqref{factor2}, we obtain Theorem \ref{t1} for all genera: $$\mathsf Z_{C, L, E}=\mathsf A^{\chi(C, \mathcal O_C)} \cdot {\mathsf B}^{\deg L}\cdot \mathsf C^{\deg E}=(1-q)^{-\chi(\mathcal O_C)} \cdot (1+qy)^{\chi(E\otimes L)}.$$ This completes the argument. \qed
\begin{example} Theorem \ref{t1} in higher genus immediately implies $$\chi\left(\quot_d (E), \wedge^k L^{[d]}\right)=0\,\,\text{ if }d\geq k+g, \,\,g\geq 1.$$ This follows by examining the coefficient of $q^d y^k$ in the expression $(1-q)^{-\chi(\mathcal O_C)}(1+qy)^{\chi(E\otimes L)}.$ 
\end{example} 
	\section{Symmetric Powers} \label{s3}
	\subsection{Genus zero.} Theorem \ref{t4} concerns the symmetric powers of the tautological bundles $\mathsf{Sym}_y L^{[d]}$ in genus $0$ and is proven in a similar fashion as Theorem \ref{t1}. The calculations are however more involved. The higher genus case and Theorem \ref{t5} will be considered in Section \ref{univf}. \vskip.1in
	
	By Section \ref{arb}, for each $d$ and $k$, the Euler characteristic of $$\chi\left(\quot_d, \mathsf{Sym}^{k} L^{[d]}\right)$$ depends polynomially on $\ell$. To prove Theorem \ref{t4}, it suffices to assume $b_i=\ell+a_i+1\ge d+1$ for all $i$. 
			
	By Hirzebruch-Riemann-Roch followed by Atiyah-Bott localization, we calculate \begin{equation}\label{syma}\chi\left(\quot_d, \mathsf{Sym}_y L^{[d]}\right)=\int_{\quot_d} \ch (\mathsf{Sym}_y L^{[d]})\, \td \left(\quot_d\right)=\sum_{\vec{d}}\int_{\fix_{\vec{d}}}
		\ch(\mathsf {Sym}_y L^{[d]})\frac{\td(\quot_d)}{e_{\mathbb{C}^*}(\textrm {N}_{\vec{d}})}\bigg|_{\fix_{\vec{d}}}.\end{equation}
Instead of Lemma \ref{chern_classes}, for the current computation we use the expression 
	\begin{equation}\label{symmm}
			\ch(\mathsf{Sym}_y L^{[d]} )\bigg|_{\textrm {F}_{\vec{d}}}=\prod_{i\in [N]} \bigg(\frac{\alpha_i(z_i-y)}{z_i(\alpha_i-y)}\bigg)^{a_i+\ell+1}\bigg(\frac{z_i}{z_i-y} \bigg)^{d_i}. 
		\end{equation} 

The Todd genera and the normal bundle contributions are found in \eqref{toddeufinal}. We substitute \eqref{toddeufinal} and \eqref{symmm} into \eqref{syma} and apply Lagrange-B\"urmann. Carrying out these steps carefully, we arrive at the following. Consider the polynomial \[P(z)=(z-y)R(z)-qz, \] and let $z_1, \ldots, z_{N+1}$ be its roots with $z_i(q=0)=\alpha_i$ for $1\leq i\leq N.$ Then, just as in the derivation leading up to \eqref{good} for exterior powers, \eqref{syma} turns into 
\begin{align*}
		(-1)^{(N-1)d}\left[q^d\right]\prod_{i\in [N]}\bigg(\frac{\alpha_i-z_{N+1}}{\alpha_i-y}\bigg)^{b_i}\bigg(\frac{z_i}{z_i-y} \bigg)^{-1} \frac{z_i^{d+1}}{P'(z_i)}\prod_{i, j\in [N],\,\,\, i\neq j}(z_i-z_j)\bigg{|}_{\epsilon=0}.
\end{align*} This simplification makes use of the fact that $$\frac{dq}{dh_i}=P'(z_i).$$ 
As in \eqref{de}, the above expression can be recast as the quotient of determinants 
\begin{align*}
	\left[\epsilon^0q^{d}\right] \frac{(-1)^{(N-1)d}}{\det(z_i^{N-j+1})}\begin{vmatrix}
		(z_1-y)z_1^{d+N-1}&\cdots& (z_1-y)z_1^{d}\\
		(z_2-y)z_2^{d+N-1}&\cdots &(z_2-y)z_2^{d}\\
		\vdots&\cdots & \vdots\\
		(z_N-y)z_N^{d+N-1}&\cdots &(z_N-y)z_N^{d}
	\end{vmatrix}\prod_{i\in [N]}\bigg( \frac{\alpha_i-z_{N+1}}{\alpha_i-y}\bigg)^{b_i}.
\end{align*}
The same derivation that led to \eqref{de2} yields the enlarged $(N+1)\times (N+1)$ determinant 
	\begin{align*}
	[\epsilon^0 q^d]\frac{(-1)^{(N-1)d}}{\det(z_i^{N-j+1})}
	\begin{vmatrix}
		(z_1-y)z_1^{d+N-1}& (z_1-y)z_1^{d+N-1}&\cdots & (z_1-y)z_1^{d}& \prod_{i=1}^{N}(\frac{\alpha_i-z_1}{\alpha_i-y})^{b_i}\\
		(z_2-y)z_2^{d+N-1}& (z_2-y)z_2^{d+N-1}&\cdots & (z_2-y)z_2^{d}& \prod_{i=1}^{N}(\frac{\alpha_i-z_2}{\alpha_i-y})^{b_i}\\
		\vdots& \vdots &\cdots & \vdots& \vdots
		\\
		(z_{N+1}-y)z_{N+1}^{d+N-1}& (z_{N+1}-y)z_{N+1}^{d+N-1}&\cdots & (z_{N+1}-y)z_{N+1}^{d}& \prod_{i=1}^{N}(\frac{\alpha_i-z_{N+1}}{\alpha_i-y})^{b_i}
	\end{vmatrix}.
\end{align*} This uses $b_i\geq d+1$ for all $i$, and the fact that $\alpha_i-z_i$ has no free $q$-term, so in particular the first $N$ entries of the last column do not contribute to the $q^d$-coefficient. 

The expression above is symmetric in the roots of $P$, and as previously remarked the substitution $\alpha_i=1$ is allowed to obtain the coefficient of $\epsilon^0$. Thus $z_1, \ldots, z_{N+1}$ become roots of $$P(z)=(z-1)^N(z-y)-qz.$$ This also turns the last column into the vector with entries $$\frac{(1-z_i)^{\chi}}{(1-y)^{\chi}}=\frac{1}{(1-y)^{\chi}}\sum_{\ell=0}^{\chi} \binom{\chi}{\ell} (-1)^{\ell} z_i^{\ell}.$$ Here $\chi=\sum_i b_i=\chi(E\otimes L)$. 

Using the additivity of the determinant with respect to the first $N$ columns, we split the last determinant into a sum \begin{align*}
\left[q^d\right]	\sum_{\ell=0}^{\chi} \sum_{m=0}^{N}\frac{(-1)^{(N-1)d+\ell}}{(1-y)^\chi} \binom{\chi}{\ell}(-y)^{m}\frac{1}{\det(z_i^{N-j+1})}
	\begin{vmatrix}
		z_1^{d+N}&\cdots& z_1^{d+m+1}&z_1^{d+m-1}&\cdots &z_1^{d}& z_1^\ell\\
		z_2^{d+N}& \cdots& z_2^{d+m+1}&z_2^{d+m-1}&\cdots & z_2^{d}& z_2^\ell\\
		\vdots& \cdots &\vdots & \vdots& \cdots & \vdots & \vdots 
		\\
		z_{N+1}^{d+N}& \cdots& z_{N+1}^{d+m+1}&z_{N+1}^{d+m-1}&\cdots & z_{N+1}^{d}& z_{N+1}^\ell
	\end{vmatrix}.
\end{align*} 
Indeed, from each of the first $N$ columns we select $N$ powers of $z_i$ whose exponents range from $d$ to $d+N$. Exactly one value $d+m$ must be skipped, giving a term in the sum. The contribution $(-y)^m$ comes from terms with exponents between $d$ and $d+m-1$. 

Regarding the last sum, we make the following three remarks.
\begin{itemize}
\item [(i)] When $\ell<d$, the above quotient of determinants is the Schur polynomial for the partition $\lambda=(d^{N-m},(d-1)^m, \ell)$. Using Jacobi-Trudi as in Lemma \ref{l4}, we obtain that $$\left[q^d\right] s_{\lambda}(z_1, \ldots, z_{N+1})=\begin{cases} (-1)^{(N-1)d} &\text{ if } \ell=m=0\\ 0 &\text{ otherwise }\end{cases}.$$ 
\item [(ii)] When $\ell>d+N$, the shape of the partition changes to $\lambda=(\ell-N,(d+1)^{N-m},d^m)$, and we also acquire an additional $(-1)^N$ coming from permuting the columns to bring the last one to the front. Note that $\lambda$ contains the rectangular partition $(d^{N+1})$ and a hook partition $\mu:=(\ell-N-d,1^{N-m})$. Examining the determinant, we can factor $z_i^{d}$ from each column. Thus  \[s_{\lambda}=e_{N+1}^{d}\cdot s_{\mu}=y^d \cdot s_{\mu}. \] Here we used that $e_{N+1}=y$ which can be seen from the expression $P(z)=(z-1)^N(z-y)-qz$.  
\item [(iii)] Finally, for $d\leq \ell\leq d+N$, the only value that can contribute is $\ell=d+m$, in which case we can directly evaluate the corresponding quotient of determinants to be $(-1)^{m}y^d$. The coefficient of $q^d$ vanishes in this case (for $d\neq 0$). 
\end{itemize}

\noindent Putting everything together we conclude $$\chi\left(\quot_d, \mathsf{Sym}_y L^{[d]}\right)=\frac{1}{(1-y)^{\chi}}+O(y^d).$$ Consequently, for $d>k$, we have $$\chi \left(\quot_d, \mathsf{Sym}^k L^{[d]}\right)=\left[y^k\right]\frac{1}{(1-y)^{\chi}}=\binom{\chi+k-1}{k}.$$ The result is also correct for $d=k$; this can be seen for instance from the result below. \qed
\vskip.1in
With a bit more effort, the same ideas (combined with a residue calculation) yield a general expression in genus $0$. We need this result in order to prove Theorem \ref{t5} in all genera in Section \ref{univf}. 
	\begin{theorem}\label{t10} When $C=\mathbb P^1$ and $\chi=\chi(E\otimes L)$, we have $$
	\chi(\quot_d(E), \mathsf{Sym}_y L^{[d]}) = \sum_{k=0}^{d}\binom{-\chi+d(N+1)}{k}\frac{(-y)^k}{(1-y)^{d(N+1)}}.$$
	\end{theorem}
\proof Since both sides depend polynomially on $\ell$, see for instance the arguments in Section \ref{arb} for the left hand side, we may assume $\ell$ is sufficiently large. In this case, we have seen above that $$\chi(\quot_d(E), \mathsf{Sym}_y L^{[d]}) =\frac{1}{(1-y)^{\chi}}+\sum_{\ell>d+N}^{\chi} \frac{1}{(1-y)^\chi} (-1)^{(N-1)d+N+\ell}\binom{\chi}{\ell} y^d \left[q^d\right] \sum_{m=0}^{N} (-y)^{m} s_{\mu(\ell, m)},$$ for the partition $\mu(\ell, m)=(\ell-N-d, 1^{N-m}).$ 

Lemma \ref{l11} below evaluates the sum over $m$. We obtain  
\begin{align*}
	\chi(\mathsf {Sym}_y L^{[d]})&=\frac{1}{(1-y)^\chi}\bigg[1+\sum_{\ell>d+N}^{\chi}(-1)^{(N-1)d+\ell} \binom{\chi}{\ell} y^{d+1}\left[t^{\ell-N-d}\right]\frac{t^{N(d-1)+1}}{(1-t)^{Nd}(1-yt)^{d+1}} \bigg] \\ &=\frac{1}{(1-y)^\chi}\bigg[1+\sum_{\ell>d+N}^{\chi}(-1)^{(N-1)d+\ell} \binom{\chi}{\ell} y^{d+1}\text{Res}_{t=0}\,\,\frac{t^{(N+1)d-\ell}}{(1-t)^{Nd}(1-yt)^{d+1}}\,dt \bigg].\end{align*} We can allow all values $\ell\geq 0$ in the sum above since the residue vanishes in the range $\ell\leq N+d$. 
The binomial theorem evaluates the sum over $\ell$. 
Letting $$\omega = \frac{t^{(N+1)d-\chi}(1-t)^{\chi-Nd}}{(1-yt)^{d+1}}\,dt,$$ we conclude that  
\begin{align*}
		\chi(\mathsf {Sym}_y L^{[d]})=\frac{1}{(1-y)^\chi}\bigg[1+(-1)^{(N-1)d+\chi}y^{d+1}\text{Res}_{t=0}\, \,\omega \bigg].
\end{align*} Lemma \ref{l12} finishes the proof.  
\qed

\begin{lemma}\label{l11} Let $z_1, \ldots, z_{N+1}$ be the roots of $P(z)=(z-1)^N(z-y)-qz.$
	For $\ell>0$, we have
	\begin{align*}
		\left[q^d\right]\sum_{m=0}^N(-y)^m s_{(\ell,1^{N-m})}(z_1, \ldots, z_{N+1})=&(-1)^N\left[t^\ell\right]\frac{yt^{N(d-1)+1}}{(1-t)^{Nd}(1-yt)^{d+1}}.
	\end{align*}
\end{lemma}
\begin{proof}
	Using Jacobi-Trudi, the left hand side of the expression in the lemma equals the $\ell\times \ell$ determinant
	\begin{align*}
		\sum_{m=0}^{N}(-y)^m\begin{vmatrix}
			e_{N+1-m}& e_{N+2-m}&e_{N+3-m}&\cdots & e_{N+\ell-m}\\
			e_0&e_1&e_2&\cdots & e_{\ell-1}\\
			0&e_0&e_1&\cdots& e_{\ell}\\
			\vdots&\vdots&\vdots&\cdots&\vdots \\
			0&0&0 &\cdots &e_1
		\end{vmatrix}.
	\end{align*}
Summing with respect to $m$, we obtain that the $i$th term in first row becomes 
\begin{align*}
	A_i=\sum_{m=0}^{N} (-y)^m e_{N+i-m}&=\left[t^{N+i}\right](1-yt+\cdots +(-1)^Ny^Nt^N)(1+e_1t+\cdots+e_{N+1}t^{N+1})\\
	&= (-1)^{N+i}\left[t^{N+i}\right]\frac{(1-(yt)^{N+1})}{1-yt}\cdot t^{N+1}P(1/t).
\end{align*} 
Expanding with respect to the first row, we obtain the required determinant equals
\begin{align*}
	A_1h_{\ell-1}-A_{2}h_{\ell-2}+\cdots+(-1)^{\ell-1} A_{\ell}
\end{align*}
where $h_j=s_{(j)}$ is the homogeneous symmetric polynomial. We know that the homogeneous symmetric polynomials are given by 
\begin{align*}
	h_i&=\left[t^i\right]\frac{1}{(1-z_1t)\cdots (1-z_{N+1}t)}= \left[t^i\right]\frac{1}{t^{N+1}P(1/t)}.
\end{align*}
Thus the required sum equals
\begin{align}\label{reqsum}
	\sum_{i=1}^{\ell}(-1)^{i-1}A_ih_{\ell-i}&= \sum_{i=1}^{\ell} (-1)^{N+1}\bigg[\left[t^{N+i}\right]\frac{(1-(yt)^{N+1})}{1-yt}t^{N+1}P(1/t)\bigg]\bigg[\left[t^{\ell-i}\right]\frac{1}{t^{N+1}P(1/t)}\bigg]\\
	&\approx\sum_{j=0}^{N} (-1)^{N}\bigg[\left[t^{N-j}\right]\frac{(1-(yt)^{N+1})}{1-yt}t^{N+1}P(1/t)\bigg]\bigg[\left[t^{\ell+j}\right]\frac{1}{t^{N+1}P(1/t)}\bigg]\nonumber\\
	&\approx \sum_{j=0}^{N} (-1)^{N}\bigg[\left[t^{N-j}\right]\frac{t^{N+1}P(1/t)}{1-yt}\bigg]\bigg[\left[t^{\ell+j}\right]\frac{1}{t^{N+1}P(1/t)}\bigg],\nonumber
\end{align} where $\approx$ means equality of the $q^d$ coefficients. 
To justify the second line, we note that the difference with the previous term equals $$\left[q^d\right](-1)^{N}\left[t^{N+\ell}\right]\bigg( \frac{1-(yt)^{N+1}}{1-yt}t^{N+1}P(1/t)\cdot \frac{1}{t^{N+1}P(1/t)} \bigg)=0$$ for $d>0$. Moreover, since $j$ runs from $0$ to $N$, we may also ignore the term $(yt)^{N+1}$ in the second line, thus yielding the third equality. 

Note that \[t^{N+1}P(1/t)=(1-yt)(1-t)^N-qt^N. \]
Thus \begin{align*}
	\left[t^{N-j}\right]\frac{t^{N+1}P(1/t)}{1-yt}=\left[t^{N-j}\right] \left((1-t)^{N} - \frac{qt^{N}}{1-yt}\right)=\begin{cases}
		(-1)^{N-j}\binom{N}{N-j}& \text{if }j>0\\
		(-1)^N-q& \text{if }j=0
	\end{cases}.
\end{align*}
Hence the $q^d$-coefficient in the sum \eqref{reqsum} equals
\begin{align*}
	\left[q^d\right]\sum_{j=0}^{N}&(-1)^j\binom{N}{j}\left[t^{\ell+j}\right]\frac{1}{(1-yt)(1-t)^N-qt^N}+(-1)^{N+1}\left[q^{d-1}\right]\left[t^\ell\right]\frac{1}{(1-yt)(1-t)^N-qt^N}\\&=(-1)^N\left[q^d\right]\left[t^{\ell+N}\right]\frac{(1-t)^N}{(1-yt)(1-t)^N-qt^N}+(-1)^{N+1}\left[q^{d-1}\right]\left[t^\ell\right]\frac{1}{(1-yt)(1-t)^N-qt^N}.
\end{align*}
We note that the order in which we take the $q^d$ and $t^{\ell+N}$-coefficients can be switched. This is allowed in our case since we are considering expressions of the form $\left(1-\mathsf A(q, t)\right)^{-1}$ expanded near $q=t=0$, where $\mathsf A$ is a polynomial in $q, t$ (and $y$). Thus, taking the respective coefficient of powers of $q$ in the above expression we obtain
\begin{align*}
(-1)^N\left[t^{\ell+N}\right]\frac{t^{Nd}}{(1-yt)^{d+1}(1-t)^{Nd}}+(-1)^{N+1}\left[t^\ell\right]\frac{t^{N(d-1)}}{(1-yt)^{d}(1-t)^{Nd}}.
\end{align*}
This immediately implies the lemma. 
\end{proof}
\begin{lemma}\label{l12}
For $\chi\geq Nd$, set $$\omega=\frac{t^{(N+1)d-\chi}(1-t)^{\chi-Nd}}{(1-yt)^{d+1}}\,dt.$$ We have $$1+(-1)^{(N-1)d+\chi}y^{d+1}\text{Res}_{t=0}\,\, \omega=\sum_{k=0}^{d}\binom{-\chi+(N+1)d}{k}\frac{(-y)^k}{(1-y)^{(N+1)d-\chi}}.$$

\end{lemma}

\proof Since $\chi\geq Nd$, the form $\omega$ has poles at worst at $t=0$, $t=\infty$ and $t=\frac{1}{y}.$ By the residue theorem, we have $$\text{Res}_{t=0}\, \omega = - \text{Res}_{t=\infty} \,\omega -\text{Res}_{t=1/y} \,\omega.$$ Changing variables $t=\frac{1}{s}$, we compute $$\text{Res}_{t=\infty}\, \omega = -\text{Res}_{s=0}\, (s-1)^{\chi-Nd}(s-y)^{-d-1}\frac{ds}{s}=(-1)^{\chi-(N+1)d}y^{-d-1} .$$
Similarly, changing variables $t=\frac{1-s}{y}$, we find \begin{align*}\text{Res}_{t=\frac{1}{y}}\, \omega &= -\text{Res}_{s=0} \,(1-s)^{-\chi+(N+1)d} (s+y-1)^{\chi-Nd} y^{-d-1} \frac{ds}{s^{d+1}}\\&=-y^{-d-1} \left[s^{d}\right] (1-s)^{-\chi+(N+1)d} (s+y-1)^{\chi-Nd}\\ &=-y^{-d-1} \sum_{k=0}^{d} (-1)^k\binom {-\chi+(N+1)d}{k}\binom{\chi-Nd}{d-k} (y-1)^{\chi-Nd-d+k}.\end{align*} Collecting terms, we obtain 
\begin{align*}1+(-1)^{(N-1)d+\chi}y^{d+1}\text{Res}_{t=0}\, \omega&=\sum_{k=0}^{d} \binom {-\chi+(N+1)d}{k}\binom{\chi-Nd}{d-k}(1-y)^{\chi-Nd-d+k}\\&=\sum_{k=0}^{d}\binom{-\chi+(N+1)d}{k}\frac{(-y)^k}{(1-y)^{(N+1)d-\chi}}.\end{align*}

To justify the last equality, we write $u=-\chi+(N+1)d$ and show more generally 
$$\sum_{k=0}^{d} \binom {u}{k}\binom{-u+d}{d-k} (1-y)^{k}=\sum_{k=0}^{d}\binom{u}{k}{(-y)^k}.$$ This follows by induction on $d$. Indeed, write $\mathsf {L}_d$ for the left hand side. Using Pascal's identity and then rewriting the binomials, we obtain \begin{align*}\mathsf{L}_{d+1}-\mathsf{L}_{d}&=\sum_{k=0}^{d+1} \binom{u}{k}\left(\binom{-u+d+1}{d+1-k}-\binom{-u+d}{d-k}\right)(1-y)^{k}=\sum_{k=0}^{d+1} \binom{u}{k} \binom{-u+d}{d+1-k} (1-y)^k\\&=\sum_{k=0}^{d+1}\binom{u}{d+1}\binom{d+1}{k}(-1)^{d-k+1}(1-y)^{k}=\binom{u}{d+1}(-y)^{d+1}.\end{align*} The proof follows immediately from here. 
\qed

\subsection{Universal functions} \label{univf} Over a smooth projective curve $C$ of arbitrary genus, let $$\mathsf{W}=\sum_{d=0}^{\infty} q^d \chi\left(\quot_d(E), \mathsf{Sym}_y L^{[d]}\right).$$ The arguments in Section \ref{arb} exhibit $\mathsf W$ as a product of universal series \footnote{Strictly speaking, we only explained the factorization $\mathsf W=\mathsf A_1^{\chi(\mathcal O_C)}\cdot \mathsf B_1^{\deg E}\cdot \mathsf B_2^{\deg L}$ in terms of $3$ universal series. An argument of \cite {Sta2} shows that only $2$ series are needed. Indeed, tensorization by a line bundle $M\to C$ gives an isomorphism $\quot_d(E)\simeq \quot_d(E\otimes M)$ in such a fashion that $L^{[d]}$ gets identified with $\left(L\otimes M^{-1}\right)^{[d]}$. On the level of generating series this implies $\mathsf B_1^{N}=\mathsf B_2$, which then yields the result with $\mathsf A=\mathsf A_1\cdot \mathsf B_1^{-N},$ $\mathsf B=\mathsf B_1$.}\begin{equation}\label{wab}\mathsf W=\mathsf A^{\chi(\mathcal O_C)} \cdot \mathsf B^{\chi(E\otimes L)}.\end{equation} In principle Theorem \ref{t10} determines both series $\mathsf A, \mathsf B$ from the genus $0$ answer. Theorem \ref{t5} asserts that more precisely we have $$\mathsf B=f\left(\frac{qy}{(1-y)^{N+1}}\right)$$ where $f(z)$ is the solution to the equation $$f(z)^N-f(z)^{N+1}+z=0,\,\quad f(0)=1.$$
{\it Proof of Theorem \ref{t5}.} The function $f$ is most conveniently expressed in terms of a change of variables. We have $$f(z)=\frac{1}{1+t}\quad  \text{ for } z=-\frac{t}{(1+t)^{N+1}}.$$ 

We record the one-variable version of the general Lagrange-B\"urmann formula \eqref{formlb}. Assuming $\Phi(0)\neq 0$, for the change of variables $z=\frac{t}{\Phi(t)}$,
 the following
general identity holds \begin{equation}\label{lb}\sum_{d=0}^{\infty} z^d\cdot \left(\left[t^d\right] \Phi(t)^{d}\cdot \Psi(t)\right)=\frac{\Psi(t)}{\Phi(t)}
 \cdot \frac{dt}{dz}\, .\end{equation}

We introduce two functions which will be useful in the argument. Write \begin{equation}\label{fchidef}\mathsf F_{\chi}(z)=\sum_{d=0}^{\infty} z^d \binom{-\chi+(N+1)d}{d}\implies F_\chi(z)=\sum_{d=0}^{\infty} z^d \left(\left[t^d\right] (1+t)^{-\chi+(N+1)d}\right).\end{equation} An immediate application of \eqref{lb} yields \begin{equation}\label{fchi}\mathsf F_{\chi}(z)=\frac{(1+t)^{-\chi+1}}{1-Nt} \quad \text {for }z=\frac{t}{(1+t)^{N+1}}.\end{equation} Setting $\chi=0$ and integrating, we also obtain the expression \begin{equation}\label{gg}\mathsf G(z)=\sum_{d=1}^{\infty} z^{d}\cdot \frac{N}{d} \binom{(N+1)(d-1)}{d-1}=1-\frac{1}{(1+t)^N},\end{equation} for the same change of variables. With this understood, we note that for the function $f$ in the theorem, we have $$f(-z)^N= 1- \mathsf G(z).$$ The statement to be proven thus becomes $$\mathsf B^{N} = 1- \mathsf G\left(-\frac{qy}{(1-y)^{N+1}}\right)$$ or equivalently \begin{equation}\label{bb}\mathsf B^N=1+\sum_{d=1}^{\infty} (-1)^{d+1} \frac{N}{d} \cdot \binom{(N+1)(d-1)}{d-1}\cdot \left(\frac{qy}{(1-y)^{N+1}}\right)^{d}.\end{equation}

Turning to the generating series \eqref{wab}, we specialize to genus $0$ and we keep track on the dependence on $\deg L=\ell$ in the notation, so that $$\mathsf W_\ell=\sum_{d=0}^{\infty} q^d\chi\left(\quot_d(E), \mathsf{Sym}_y L^{[d]}\right)=\mathsf A^{-1}\cdot \mathsf B^{\chi}.$$ As usual, $\chi=\chi(E\otimes L)$. This yields \begin{equation}\label{identity}\mathsf W_{\ell+1}=\mathsf W_{\ell} \cdot \mathsf B^{N}.\end{equation}
By Theorem \ref{t10}, we have $$\mathsf W_\ell=\sum_{d=0}^{\infty} \mathsf {c}_d(\chi) \cdot q^d, \quad \mathsf W_{\ell+1}=\sum_{d=0}^{\infty} \mathsf {c}_d(\chi+N) \cdot q^d,$$ where for simplicity, we wrote \begin{equation}\label{ad} \mathsf{c}_{d}(\chi)= \sum_{k=0}^{d}\binom{-\chi+d(N+1)}{k}\frac{(-y)^k}{(1-y)^{d(N+1)}}.\end{equation} 
Examining the coefficient of $q^d$ in the identity \eqref{identity}, it follows that in order to confirm \eqref{bb} it suffices to prove $$\mathsf{c}_{d}(\chi+N)=\mathsf{c}_d(\chi)+\sum_{\ell=1}^{d} \mathsf{c}_{d-\ell}(\chi)\cdot (-1)^{\ell+1}\frac{N}{\ell} \binom{(N+1)(\ell-1)}{\ell-1}\left(\frac{y}{(1-y)^{N+1}}\right)^{\ell}.$$

We use the defining expressions \eqref{ad} to verify this equality. After multiplying by $(1-y)^{d(N+1)}$ and extracting the coefficient of $y^{k}$ on both sides, we need to show that for $0\leq k\leq d$, we have \begin{equation}\label{down}\binom{-\chi-N+d(N+1)}{k}=\binom{-\chi+d(N+1)}{k}-\sum_{\ell=1}^k \frac{N}{\ell} \binom{(N+1)(\ell-1)}{\ell-1} \binom{-\chi+(d-\ell)(N+1)}{k-\ell}.\end{equation} Using Pascal's identity, it is easy to see that if \eqref{down} holds for $k$ and all $\chi$, then it also holds for $k-1$ and all $\chi$. Thus, by downward induction it suffices to assume $k=d$. In this case, we seek to show $$\sum_{\ell=1}^d \frac{N}{\ell} \binom{(N+1)(\ell-1)}{\ell-1}\binom{-\chi+(d-\ell)(N+1)}{d-\ell}=\binom{-\chi+d(N+1)}{d}-\binom{-\chi-N+d(N+1)}{d}.$$ This is indeed correct. Recalling \eqref{fchidef} and \eqref{gg}, we see that the two sides equal the $z^d$-coefficient in the identity 
$$\mathsf G(z)\cdot \mathsf F_{\chi}(z)=\mathsf F_{\chi}(z)-\mathsf F_{\chi+N}(z).$$ The latter equality is immediately justified using the explicit formulas \eqref{fchi} and \eqref{gg} after changing variables from $z$ to $t$ as above.

	\section{Higher rank quotients}\label{s4} It is natural to wonder whether the formulas proven in Section \ref{section2} extend to Quot schemes parametrizing quotients of rank $r>0$. We write $\quot_d(E, r)$ for the corresponding Quot scheme. We restrict to the case $C=\mathbb P^1$, and $E$ a trivial rank $N$ vector bundle. In this case, $\quot_d(E, r)$ is smooth. 
	
	Theorem \ref{t3} gives an expresion for the Euler characteristics
	\begin{align*}
			\chi(\quot_d(E, r),\wedge_y L^{[d]})
	\end{align*}
in terms of the roots $z_1,\dots,z_N$ of the equation $$(z-1)^N-q(z+y)z^{r-1}=0.$$ Unlike our previous computations, this result does not imply the answer in higher genus, since the universality arguments in Section \ref{arb} fail in this case. 
\vskip.1in
\noindent{\it Proof of Theorem $3$.} We follow the same steps as in the proof of Theorem \ref{t1}. However, some modifications are necessary. Write $$s=N-r$$ for the rank of the subbundles. We use the torus action on $\quot_d(E, r)$ coming from the torus action on $E=\mathbb C^N\otimes \mathcal O_{\mathbb P^1}$ with weights $-w_1,\ldots, -w_N$. The fixed loci are parameterized by pairs $(\vec{d}, I)$, where $\vec d=(d_1, \ldots, d_s)$ with $|\vec{d}|=d_{1} + \cdots + d_{s}=d$, and $I\subset [N]$ is a subset of cardinality $s$. The fixed loci are products of symmetric powers of $\mathbb P^1$: \[\fix_{\vec{d}, I}=\mathbb{P}^{d_1}\times\cdots\times \mathbb{P}^{d_s}.\] Each such product appears $\binom{N}{s}$ times corresponding to the choice $I$ of $s$ summands of the trivial bundle $E$ into which the kernel injects $$0\to S\to \bigoplus_{i\in I} \mathcal O_{\mathbb P^1}\to \bigoplus_{i\in [N]} \mathcal O_{\mathbb P^1}=E.$$ This changes slightly the expressions for the normal bundles $\textrm{N}_{\vec d, I}$. With the same notation as in Section \ref{torusact}, we find that \eqref{movingpart} becomes \begin{align*}
		\textrm{N}_{\vec d, I}= \bigoplus_{i\in I,\,\,j\in [N],\,\, i\ne j}\pi_*\left(\mathcal{K}_i^{\vee}\right)-\bigoplus_{i,j\in I,\,\,i\ne j}\pi_*\left(\mathcal{K}^{\vee}_i \otimes\mathcal{K}_j\right) .
	\end{align*} Compared to prior expressions, the range of the indices $i, j$ has changed. 
	The above sheaves carry weights $w_{i}-w_j$. By direct calculation, we obtain the analogue of equation \eqref{toddeufinal}
	$$	\frac{\td (\quot_d)}{{e}_{\mathbb{C}^*}(\textrm{N}_{\vec d, I})}=(-1)^{(s-1)d}\prod_{i\in I}\bigg(\frac{h_iz_i^{N-s}}{R(z_i)} \bigg)^{d_i+1} z_i^{ d+1}
		\prod_{i,j\in I,\, i\ne j}(z_i-z_j),$$ where $z_i$, $\alpha_i$ and $R(z)=\prod_{j=1}^{N} (z-\alpha_j)$ were defined before. Similarly, the analogue of Lemma \ref{chern_classes} becomes \begin{align*}
				\ch(\wedge_yL^{[d]})\bigg|_{\fix_{\vec{d},I}}=\prod_{j\in[N]}\bigg(1+\frac{y}{\alpha_j} \bigg)^{\ell+1}\prod_{i\in I}\bigg(1+\frac{y}{z_i} \bigg)^{d_i-(\ell+1)}. 
		\end{align*}
We substitute these expressions into Hirzebruch-Riemann-Roch and Atiyah-Bott localization: 
\begin{align*}
	&\chi(\quot_d,\wedge_y L^{[d]}) = \sum_{\vec{d}, I}\int_{\fix_{\vec{d},I }}
	\ch(\wedge_y L^{[d]})\,\frac{\td (\quot_d)}{\mathsf {e}_{\mathbb{C}^*}(\textrm{N}_{\vec d, I})}\,\bigg{|}_{\fix_{\vec{d},I }}.\end{align*}
After substitution, we invoke Lagrange-B\"urmann formula, but the change of variables takes a slightly different form for $r>0$. If  \[P(z)=R(z)-q(z+y)z^{r-1}\]
and $z_i$ are the roots of the above equation, the reader can verify by direct calculation that we arrive at the following expression
$$\left[q^{d}\right] (-1)^{(s-1)d} \,\frac{\prod_{j=1}^{N} (\alpha_j+y)^{\ell+1}} {\prod_{j=1}^{N} \alpha_j^{\ell+1}}\, \cdot\sum_{\vec{d}, I}\,\prod_{i\in I}(z_i+y)^{-(\ell+1)}\frac{z_i^{d+\ell+N-s+1}}{P'(z_i)}\prod_{i, j\in I, \,i\ne j}(z_i-z_j)\bigg{|}_{\epsilon=0}.$$
Since $P(-y)=R(-y)$, it follows that $$\prod_{j=1}^{N} (\alpha_j+y)=\prod_{j=1}^{N} (z_j+y),$$ and thus the previous expression can be further simplified to \begin{equation}\label{prelap}
\left[q^{d}\right]\frac{(-1)^{(s-1)d}}{\prod_{j=1}^{N} \,\,\alpha_j^{\ell+1}}\,\,\sum_{\vec{d}, I} \prod_{j\not \in I}(z_j+y)^{\ell+1}\,\,	\prod_{i\in I}\frac{z_i^{d+\ell+N-s+1}}{P'(z_i)}\prod_{i, j\in I, \,i\ne j}(z_i-z_j)\bigg{|}_{\epsilon=0}.
\end{equation}

Compared to Theorem \ref{t1}, for $r>0$ we do not have an additional root, so we finish the argument in a different way. The key observation is that we can rewrite \eqref{prelap} as the $q^d$-coefficient in the quotient of two $N\times N$ determinants
\begin{equation}\label{debig}
\frac{(-1)^{(s-1)d}} {\prod_{j=1}^{N} \alpha_j^{\ell+1}}\frac{1}{\det (z_i^{N-j})}
\left|\begin{array}{cccc}
z_1^{d+\ell+N} & z_2^{d+\ell+N} & \cdots & z_N^{d+\ell+N}\\
		z_1^{d+\ell+N-1}&z_2^{d+\ell+N-1}&\cdots&z_N^{d+\ell+N-1} \\
	\vdots&\vdots & \cdots & \vdots \\
		z_1^{d+\ell+N-s+1}&z_2^{d+\ell+N-s+1}&\cdots&z_N^{d+\ell+N-s+1} \\	
		\\	
		z_1^{N-s-1}(z_1+y)^{\ell+1}&z_2^{N-s-1}(z_2+y)^{\ell+1}&\cdots &z_N^{N-s-1}(z_N+y)^{\ell+1}\\
		\vdots&\vdots & \cdots & \vdots \\
		(z_1+y)^{\ell+1}		&(z_2+y)^{\ell+1}&\cdots &(z_N+y)^{\ell+1}
	\end{array}
	\right|.
\end{equation} This agrees with the prior expression \eqref{prelap}. To justify this assertion, we use generalized Laplace expansion of the determinant in the numerator along the last $N-s$ rows simultaneously. Picking $s$ columns labeled by the index set $I$, the corresponding minors of \eqref{debig} (evaluated using Vandermonde determinants) contribute exactly the $I$-term of \eqref{prelap}. 

Since the above expression is a symmetric function in the $z_i$'s, it can also be expressed in terms of the elementary symmetric functions, which are in turn polynomials in $\alpha_j$. Setting $\epsilon=0$ corresponds to setting $\alpha_j=1$ for all $j$, or in turn working with the roots of $$P(z)=(z-1)^N-q(z+y)z^{r-1}=0.$$ This completes the argument. \qed
\vskip.1in
\begin{corollary}\label{thm:high_rank_det}
		As in Theorem \ref{t3}, let $C=\mathbb P^1$, $E$ is trivial of rank $N$, $\deg L=\ell\geq -d-1$. Then
		\begin{align*}
				\chi\left(\quot_d(E, r) ,\det L^{[d]}\right) =(-1)^{(N-r-1)d}\left[q^{d}\right]s_{\lambda}(z_1,z_2,\dots z_N)
		\end{align*}
	for the partition $\lambda=((d+\ell+1)^{N-r})$. The $z_i$'s are the distinct roots of the equation $(z-1)^N-qz^{r-1}=0.$ 
	\end{corollary}
	
\proof The statement follows by running the argument above for the determinant $\det L^{[d]}$ instead of $\wedge_y L^{[d]}$. (We prefer this route since extracting the top $y$-coefficient in the determinant \eqref{debig} requires some care.) The reader can verify that this results in the following two changes:
\begin{itemize}
\item [(i)] the new localization sum requires a new change of variables, so in particular, $z_1, \ldots, z_N$ are roots of $R(z)-z^{r-1}q=0$; 
\item [(ii)] the analogue of expression \eqref{prelap} is
$$\left[q^{d}\right]\frac{(-1)^{(s-1)d}}{\prod_{j=1}^{N} \alpha_j^{\ell+1}}\,\sum_{\vec{d}, I} \,\,\prod_{i\in I}\frac{z_i^{d+\ell+N-s+1}}{P'(z_i)}\,\,\prod_{i, j\in I, \,i\ne j}(z_i-z_j)\bigg{|}_{\epsilon=0}.$$
\end{itemize}
\noindent As a result, the counterpart of \eqref{debig} takes the form 
$$\left[\epsilon^0q^d\right]\frac{(-1)^{(s-1)d}} {\prod_{j=1}^{N} \alpha_j^{\ell+1}}\frac{1}{\det (z_i^{N-j})}
\left|\begin{array}{cccc}
z_1^{d+\ell+N} & z_2^{d+\ell+N} & \cdots & z_N^{d+\ell+N}\\
		z_1^{d+\ell+N-1}&z_2^{d+\ell+N-1}&\cdots&z_N^{d+\ell+N-1} \\
	\vdots&\vdots & \cdots & \vdots \\
		z_1^{d+\ell+N-s+1}&z_2^{d+\ell+N+1-s}&\cdots&z_N^{d+\ell+N-s+1} \\	
		\\	
		z_1^{N-s-1}&z_2^{N-s-1}&\cdots &z_N^{N-s-1}\\
		\vdots&\vdots & \cdots & \vdots \\
		1	&1&\cdots &1
	\end{array}
	\right|.$$
The proof is completed setting $\alpha_j=1$, and noting that when $d+\ell+1\geq 0$, the last expression is exactly the Schur polynomial of the partition $\lambda=((d+\ell+1)^s)$. \qed

\begin{example} In the simplest case $d=0$, the Quot scheme is the Grassmannian $G=G(s, N)$ and $\det L^{[d]}=\mathcal O_G(\ell+1)$. In the corollary, since we are extracting the coefficient of $q^0$, we can set $z_1=z_2=\ldots=z_N=1$ to obtain the correct identity $$\chi(G, \mathcal O_{G}(\ell+1))=s_{\lambda}(1, \ldots, 1), \quad \lambda=((\ell+1)^s).$$ 
\end{example}

\begin{example} \label{ex16} An interesting specialization of Corollary \ref{thm:high_rank_det} arises for $\ell=0$. We show 
		\begin{align*}
		\chi\left(\quot_d(E, r) ,\det \mathcal O^{[d]}\right) =\binom{N}{r+d}.
		\end{align*}
We have \[\chi(\quot_d, \det\mathcal O^{[d]})=\left[q^d\right](-1)^{(s-1)d}s_\lambda(z_1,\dots,z_N) \]
where $\lambda=((d+1)^s)$. The elementary symmetric functions in $z_1,\dots z_N$ are\begin{align*}
	e_j=\begin{cases}
		\binom{N}{j} & j\neq s+1\\
		\binom{N}{j}+(-1)^{s}q& j= s+1
	\end{cases}.
\end{align*}
Using Jacobi-Trudi, we have
\begin{align*}
	s_\lambda(z_1,\dots,z_N)=\begin{vmatrix}
		e_{s}& e_{s+1}&e_{s+2}&\cdots& e_{s+d-1}&e_{s+d}\\
		e_{s-1}&e_{s}&e_{s+1}&\cdots &e_{s+d-2}&e_{s+d-1}\\
		\vdots&\vdots&\vdots&\cdots &\vdots &\vdots\\
		e_{s-d+1}&e_{s-d+2}&e_{s-d+3}&\cdots& e_s&e_{s+1}\\
		e_{s-d}&e_{s-d+1}&e_{s-d+2}&\cdots& e_{s-1} &e_s\\
	\end{vmatrix}.
\end{align*}
In the $(d+1)\times (d+1)$ determinant, the only term yielding the power $q^d$ is $(-1)^de_{s+1}^de_{s-d}$, coming from the lower left corner $e_{s-d}$ and the terms $e_{s+1}$ above the diagonal. To conclude, it remains to note that $$\left[q^{d}\right]e_{s+1}^{d}e_{s-d}=(-1)^{sd}\binom{N}{s-d}.$$ 
\end{example}

\begin{example} Assume $d>s(\ell+1)$. The Schur polynomial $s_{\lambda}$ has weighted degree $|\lambda|=s(d+\ell+1)<(s+1)d$ in the elementary symmetric functions $e_i$, where we set $\deg e_i=i$. 
We noted in Example \ref{ex16} that only $e_{s+1}$ contains a linear $q$-term. By degree reasons, $e_{s+1}$ appears in $s_{\lambda}$ with exponent $<d$. Thus, in this case the $q^d$-coefficient vanishes, and 
$$ \chi\left(\quot_d(E, r) ,\det L^{[d]}\right)=0.$$
\end{example} 

\begin{example} Assume $d=s(\ell+1)$, so that $d+\ell+1=(s+1)(\ell+1)$ and $|\lambda|=d(s+1)$ for $\lambda=((d+\ell+1)^s)$. With these numerics, we claim that \begin{equation} \label{sces} s_{\lambda}=(-1)^{sd} e_{s+1}^{d} + \text { lower order terms in } e_{s+1}.\end{equation} Using that the only nonzero $q$-contribution in $e_j(z_1, \ldots, z_N)$ is given by $$\left[q\right] e_{s+1}(z_1, \ldots, z_{N})=(-1)^s,$$ we obtain 
$$\left[q^d\right] s_{\lambda}(z_1, \ldots, z_N)=1,\text{ and thus } \chi\left(\quot_d(E, r) ,\det L^{[d]}\right)=1.$$ 

To justify \eqref{sces}, we let $$(x_1, \ldots, x_N)=(1, \zeta, \zeta^2, \ldots, \zeta^s, 0, \ldots, 0),$$ where $\zeta$ is a primitive $(s+1)$-root of $1$. In this case, we have $$e_{s+1}(x_1, \ldots, x_N)=(-1)^s,\quad e_j(x_1, \ldots, x_N)=0 \text { for }j\neq 0,\,\, j\neq s+1.$$ Thus, to confirm \eqref{sces} it remains to show that 
\begin{equation}\label{53}s_{\lambda}(x_1, \ldots, x_N)=1.\end{equation} This follows from the (first) Jacobi-Trudi identity $$s_{\lambda}=\begin{vmatrix}
		h_{(s+1)(\ell+1)}& h_{(s+1)(\ell+1)+1}&\cdots &h_{(s+1)(\ell+1)+(s-1)}\\
		h_{(s+1)(\ell+1)-1}&h_{(s+1)(\ell+1)}&\cdots &h_{(s+1)(\ell+1)+(s-2)}\\
		\vdots&\vdots&\cdots &\vdots\\
	
		h_{(s+1)(\ell+1)-(s-2)}&h_{(s+1)(\ell+1)-(s-1)}& \cdots& h_{(s+1)(\ell+1)+1}\\
		h_{(s+1)(\ell+1)-(s-1)}&h_{(s+1)(\ell+1)-(s-2)}& \cdots& h_{(s+1)(\ell+1)}\\
	\end{vmatrix},$$ where $h_j$ are the homogeneous symmetric functions. In our case, we have $$h_j(x_1, \ldots, x_N)=1 \text{ if } j\equiv 0 \mod s+1, \quad h_j(x_1, \ldots, x_N)=0 \text{ otherwise}.$$ Hence the above matrix evaluated at $(x_1, \ldots, x_N)$ is the identity, yielding \eqref{53}. 

\end{example} 
	\section{Further questions} 
	\subsection{Cohomology groups} It is natural to inquire whether Theorem \ref{t1} can be refined to yield information about all cohomology groups of the tautological bundles $\wedge^k L^{[d]}.$ We ask:\vskip.05in
	
	\noindent
	\begin{question} Is it true that
\begin{equation}\label{spec} H^{\bullet} \left(\mathsf {Quot}_{d}(E), \wedge^{k} L^{[d]}\right)=\wedge^{k} H^{\bullet}(E\otimes L)\otimes \mathsf{Sym}^{d-k} H^{\bullet} (\mathcal O_C)?\end{equation}
\end{question}

To explain the notation, if $V^{\bullet}=V_0\oplus V_1$ is a $\mathbb Z_2$-graded vector space, we define the graded vector spaces
$$\wedge^{k} V^{\bullet}=\bigoplus_{i+j=k} \wedge^i V_0\otimes \mathsf{Sym}^{j} V_1,\quad \mathsf {Sym}^{k} V^{\bullet}=\bigoplus_{i+j=k} \mathsf{Sym}^i V_0\otimes \wedge^{j} V_1$$ where the summands have degree $j$. With the convention $$\dim\, W^{\bullet}=\sum (-1)^j \dim\, W^j$$ for the superdimension of a graded vector space, the usual formulas hold true $$\dim\, \wedge^kV^{\bullet}=\binom{\dim V^{\bullet}}{k},\quad \dim\, \mathsf{Sym}^k V^{\bullet}=(-1)^k\binom{-\dim V^{\bullet}}{k}.$$ 
Thus, taking dimensions in \eqref{spec}, we immediately match the expressions in Theorem \ref{t1}. There is also a natural analogue for Theorem \ref{t2}. The study of these questions may require understanding the derived category of $\mathsf {Quot}_d(E)$. 
\vskip.1in

\noindent {\bf Evidence.} Formula \eqref{spec} is true in the following cases
\begin{itemize}
\item [(i)] over the symmetric product of a curve, that is for $\text{rank }E=1$. This was shown in \cite[Section 3]{K2} using the derived category;
\item [(ii)] for $d=1$ so that $\mathsf {Quot}_1(E)=\mathbb P(E),$ the projective bundle of length $1$ quotients of $E$;
\item [(iii)] for $k=0$, the formula predicts the Hodge numbers $h^{p, 0} (\mathsf {Quot}_d(E))=\binom{g}{p}$ for $p\leq d$. This follows from \cite{BFP, R} which give the Hodge polynomials
$$\sum h^{p, q}(\mathsf {Quot}_d(E)) (-u)^p (-v)^q t^d=\prod_{i=0}^{\text{rk} (E)-1} \frac{(1-u^i v^{i+1} t)^{g}(1-u^{i+1}v^{i}t)^{g}}{(1-u^i v^i t)(1-u^{i+1}v^{i+1}t)}.$$
\end{itemize}

\subsection{Rationality} For any line bundles $L_1, \ldots, L_\ell\to C$ and integers $k_1, \ldots, k_{\ell}\geq 0$, the series of $K$-theoretic invariants $$\mathsf Z_{C, E}\left(L_1, \ldots, L_{\ell}\,|\,k_1, \ldots, k_{\ell}\right)=\sum_{d} q^d \chi\left(\quot_d(E), \wedge^{k_1} L_1^{[d]}\otimes \cdots \otimes \wedge^{k_{\ell}} L_{\ell}^{[d]}\right)$$ are given by rational functions with pole at $q=1$. This assertion was proved in \cite {AJLOP} in the context of punctual Quot schemes of surfaces, endowed with the virtual class, but the same argument applies here as well. (The argument proceeds by localization when $E$ is split, but this is sufficient in light of the universality statements of Section \ref{arb}.) Keeping with the theme of Section \ref{aws}, we note that a similar result also holds true for Hilbert schemes of points on surfaces (without virtual classes). This was conjectured in \cite{AJLOP} and proved in \cite {A2}. 

Theorem \ref{t1} gives the simple expression $$\mathsf Z_{C, E}(L\,|\,k)=\binom {\chi(E\otimes L)}{k} q^k \cdot (1-q)^{-\chi(\mathcal O_C)}.$$  
\begin{question}\label{expform} What is the structure of the rational functions $\mathsf Z_{C, E}(L_1, \ldots, L_{\ell}\,|\, k_1, \ldots, k_{\ell})$? Do they admit explicit formulas? \end{question}

It is natural to inquire whether the results proved here carry over to the Quot schemes $\mathsf {Quot}_d(E, r)$ parametrizing quotients of $E$ of any rank $r$ and degree $d$. The latter possess $2$-term perfect obstruction theories \cite {CFK, MO}. 

\begin{question}

Are the series $$\mathsf Z_{C, E}^{(r)} \left(L_1, \ldots, L_{\ell}\,|\,k_1, \ldots, k_{\ell}\right)=\sum_{d=0}^{\infty} q^{d} \chi^{\textrm{vir}} \left(\mathsf {Quot}_d(E, r), \wedge^{k_1} L_1^{[d]}\otimes \ldots \otimes \wedge^{k_{\ell}} L_{\ell}^{[d]}\right)$$ rational functions with pole only at $q=1$? Are there explicit expressions for  $\mathsf Z_{C, E}^{(r)} \left(L_1, \ldots, L_{\ell}\,|\,k_1, \ldots, k_{\ell}\right)$? 
\end{question}

\noindent Here, for a scheme $Y$ with a $2$-term perfect obstruction theory and virtual structure sheaf $\mathcal O_Y^{\text{vir}}$, and for $V\to Y$, we set $$\chi^{\textrm{vir}}(Y, V)=\chi(Y, V\otimes \mathcal O_Y^{\text{vir}}).$$

\subsection {Higher rank} Extending Theorems \ref{t1} to $K$-theory classes $V\to C$ of arbitrary rank is not immediate. In general, a change of variables is likely needed. 
\begin{question} Find a closed-form expression for the series $$\sum_{d=0}^{\infty} q^d \chi\left(\quot_d(E), \wedge_y V^{[d]}\right).$$ \end{question} 
\noindent This may potentially be used to address Question \ref{expform} as well. Theorem \ref{t5} partially addresses the case $V=-L$, for $L$ a line bundle.

Turning to higher rank quotients and Theorem \ref{t3}, we could ask for the arbitrary genus version:  
\begin{question} For line bundles $L\to C$, find a closed-form expression for $\chi^{\text{vir}}(\quot_d(E, r),\wedge_y L^{[d]}).$ 
\end{question}
\noindent For example, in genus $0$, for $E$ trivial of rank $N$, for rank $r>0$, $\deg L=\ell$, numerical experiments suggest that the answer stabilizes to
$$\chi\left(\quot_d(E,r), \wedge_y L^{[d]}\right)=(1+y)^{N(\ell+1)},$$ as soon as $d\geq (N-r)(\ell+1)$.  

Analogous questions can be asked about symmetric powers or other Schur functors as well.


\begin{thebibliography}{1}

\bibitem [A1]{A}

N. Arbesfeld, {\it $K$-theoretic Donaldson-Thomas theory and the Hilbert scheme of points on a surface}, Alg. Geom. 8 (2021), 587 -- 625

\bibitem [A2]{A2} 

N. Arbesfeld, {\it $K$-theoretic descendent series for Hilbert schemes of points on surfaces}, preprint, \texttt{arXiv:2201.07392}

\bibitem [AJLOP]{AJLOP}

N. Arbesfeld, D. Johnson, W. Lim, D. Oprea, R. Pandharipande, {\it The virtual $K$-theory of Quot schemes of surfaces},  J. Geom. Phys. 164 (2021), Article 104154

\bibitem [BFP]{BFP}

M. Bagnarol, B. Fantechi, F. Perroni, {\it On the motive of zero-dimensional Quot schemes
on a curve}, New York J. Math. 26 (2020), 138 -- 148

\bibitem [BD]{BD}

K. Behrend, A. Dhillon, {\it On the motivic class of the stack of bundles}, Adv. Math. 212 (2007), 617 -- 644

\bibitem [Ber1]{Ber1}

A. Bertram, {\it Towards a Schubert Calculus for Maps from a Riemann Surface to a Grassmannian}, Internat. J. Math 5 (1994), 811 -- 825

\bibitem [Ber2]{Ber2}

A. Bertram, {\it Quantum Schubert Calculus}, Adv. Math. 128 (1997), 289 -- 305

\bibitem [BDW]{BDW}
A. Bertram, G. Daskalopoulos, R. Wentworth, {\it Gromov Invariants for Holomorphic Maps from Riemann Surfaces to Grassmannians}, J. Amer. Math. Soc. 9
(1996), 529 -- 571

\bibitem [Bi]{B}

E. Bifet, {\it Sur les points fixes du sch\'ema $\textrm{Quot}_{\mathcal O_X^r/X/k}$ sous l'action du tore ${\mathbf G}_{m, k}^r$}, C. R. Acad. Sci. Paris S\'er. I Math. 309 (1989), 609 -- 612

\bibitem [BGL]{BGL}

E. Bifet, F. Ghione, M. Letizia, {\it On the Abel-Jacobi map for divisors of higher rank on a curve}, Math. Ann. 299 (1994), 641 -- 672 

\bibitem [BDH]{BDH}

I. Biswas, A. Dhillon, J. Hurtubise, {\it Automorphisms of the quot schemes associated to compact Riemann surfaces}, IMRN (2015), 1445 -- 1460

\bibitem [BGS]{BGS}

I. Biswas, C. Gangopadhyay, R. Sebastian, {\it Infinitesimal deformations of some Quot schemes}, preprint, \texttt{arXiv:2203.13150}

\bibitem [Bo]{Bo}

A. Bojko, {\it Wall-crossing for punctual Quot-schemes}, preprint, \texttt{arXiv:2111.11102}

\bibitem [BKR]{BKR}

T. Bridgeland, A. King, M. Reid, {\it The McKay correspondence as an equivalence of derived categories,} J. Amer.
Math. Soc. 14 (2001), 535 -- 554

\bibitem[C]{C}

L. Chen, {\it Poincar\'e polynomials of hyperquot schemes}, Math. Ann. 321 (2001), 235 -- 251

\bibitem [CFK]{CFK}

I Ciocan-Fontanine, M Kapranov, {\it Virtual fundamental classes via dg-manifolds},
Geom. Topol. 13 (2009), 1779 -- 1804 

\bibitem [Da]{Da}

G. Danila, {\it Sections de la puissance tensorielle du fibr\'e tautologique sur le sch\`ema de Hilbert des points d'une surface},  Bull. Lond. Math. Soc. 39 (2007), 311 -- 316

\bibitem [EGL]{EGL}

G. Ellingsrud, L. G\"ottsche, M. Lehn, {\it On the
cobordism class of the Hilbert scheme of a surface}, J. Algebraic Geom. 10 (2001), 81 -- 100 

\bibitem [Ga]{Ga}

C. Gangopadhyay, {\it Automorphisms of relative Quot Schemes}, Proc. Math. Sci. 129 (2019), Article 85

\bibitem [Ge]{G}

I. Gessel, {\it A combinatorial proof of the multivariable Lagrange inversion formula}, J. Comb. Theory, Ser. A,  45 (1987), 178 -- 195 

\bibitem [GS]{GS}

C. Gangopadhyay, R. Sebastian, {\it Nef cones of some Quot schemes on a smooth projective curve}, C. R. Math. Acad. Sci. Paris 359 (2021), 999 -- 1022

\bibitem [H]{H}

M. Haiman, {\it Hilbert schemes, polygraphs and the Macdonald positivity conjecture}, J. Amer. Math. Soc. 14 (2001), 941 -- 1006

\bibitem [HPL]{HPL}

V. Hoskins, S. Pepin Lehalleur, {\it On the Voevodsky motive of the moduli stack of vector bundles on a curve}, Q. J. Math. 72 (2021), 71 -- 114

\bibitem [K1]{K}

A. Krug, {\it Remarks on the derived McKay correspondence for Hilbert schemes of points and tautological bundles}, Math. Ann. 371 (2018), 461 -- 486

\bibitem [K2]{K2}

A. Krug, {\it Extension groups of Tautological Bundles on Symmetric Products of Curves}, Beitr. Algebra Geom., to appear

\bibitem [M]{Ma}

A. Marian, {\it On the intersection theory of Quot schemes and moduli of bundles with sections}, J. Reine Angew. Math. 610 (2007), 13 -- 27

\bibitem [MO1]{MO} 

A. Marian, D. Oprea, {\it Virtual intersections on the Quot scheme and Vafa-Intriligator formulas}, Duke Math. J. 136 (2007), 81 -- 131

\bibitem [MO2]{MarOp} 

A. Marian, D. Oprea, {\it Counts of maps to Grassmannians
and intersections on the moduli space of bundles}, J. Diff. Geom. 76 (2007), 155 -- 175

\bibitem [Mo]{M}

T. Mochizuki, {\it The structure of the cohomology ring of the filt schemes}, preprint, \texttt{arXiv:0301184}

\bibitem [O]{O}

D. Oprea, {\it Big and nef tautological vector bundles over the Hilbert scheme of points}, preprint, available at \texttt{https://mathweb.ucsd.edu/\~{}doprea/bigandnef.pdf}

\bibitem [OP]{OP}

D. Oprea, R. Pandharipande, {\it Quot schemes of curves and surfaces: virtual classes, integrals, Euler characteristics}, Geom. Topol. 25 (2021), 3425 -- 3505
\bibitem [R]{R}

A. Ricolfi, {\it On the motive of the Quot scheme of finite quotients of a locally free sheaf}, J. Math. Pures Appl. 144 (2020), 50 -- 68

\bibitem [RZ]{RZ}

Y. Ruan, M. Zhang, {\it Verlinde/Grassmannian Correspondence and Rank 2 $\delta$-wall-crossing}, preprint, \texttt{arXiv:1811.01377}

\bibitem [Sc1]{Sc}

L. Scala, {\it Cohomology of the Hilbert scheme of points on a surface with values in representations of tautological bundles}, Duke Math. J. 2 (2009), 211 -- 267. 

\bibitem [Sc2]{Sc2}

L. Scala, {\it Higher symmetric powers of tautological bundles on Hilbert schemes of points on a surface}, preprint, \texttt{arXiv:1502.07595}

\bibitem [St1]{Sta1}

S. Stark, {\it On the Quot scheme $\text{Quot}^{\ell}(\mathcal E)$}, preprint, \texttt{arXiv:2107.03991}

\bibitem [St2]{Sta2}

S. Stark, {\it Cosection localisation and the Quot scheme $\text{Quot}^{\ell}(\mathcal E)$},  preprint, \texttt{arXiv:2107.08025} 

\bibitem [Str]{St}

S. Stromme, {\it On parametrized rational curves in Grassmann varieties}, Space curves, Rocca di Papa, 1985, Lecture Notes in Math. 1266 (1987), 251 -- 272

\bibitem [T]{T}

M. Thaddeus, {\it Stable pairs, linear systems and the Verlinde formula}, Invent. Math. 117 (1994), 317 -- 353

\bibitem [WZ]{WZ}

Z. Wang, J. Zhou, {\it Tautological sheaves on Hilbert schemes of points}, J. Algebraic Geom. 23 (2014), 669 -- 692 

\end{thebibliography}
\end{document}